\documentclass[oneside,reqno]{amsart}

\usepackage{amssymb,amsfonts,amsmath,mathtools}
\usepackage{tikz-cd}
\usepackage{wasysym}
\usepackage{enumerate}
\usepackage{xcolor}
\usepackage{mathrsfs}
\usepackage{xparse}

\usepackage{comment}
\usepackage[utf8]{inputenc}
\usepackage[T1]{fontenc}
\usepackage[english]{babel}
\usepackage[a4paper]{geometry}
\usepackage[all,arc]{xy}
\usepackage{comment}
\usepackage{setspace}
\usepackage[textsize=small]{todonotes}
\def\LTC{{\sf LTC}}
\def\LTCi{{\sf LTC1_i}}

\def\LTCii{{\sf LTC2_i}}

\def\LTCiii{{\sf LTC3_i}}

\def\GSC{{\sf GSC}}
\def\SDGSC{{\sf SDGSC}}
\def\BL{\mathsf{BL}}

\newcounter{mycomment}

\swapnumbers

\usepackage[pagebackref]{hyperref}
\PassOptionsToPackage{unicode}{hyperref}
\PassOptionsToPackage{naturalnames}{hyperref}
\hypersetup{pdfauthor={Name}}
\definecolor{dark-red}{rgb}{0.5,0.15,0.15}
\definecolor{dark-blue}{rgb}{0.15,0.15,0.6}
\definecolor{dark-green}{rgb}{0.15,0.6,0.15}
\hypersetup{
    colorlinks, linkcolor=dark-red,
    citecolor=dark-blue, urlcolor=dark-green
}

\renewcommand*{\backref}[1]{}
\renewcommand*{\backrefalt}[4]{%
  \ifcase #1 %
No citations.
  \or
(cit. on p. #2).%
  \else
(cit on pp. #2).%
  \fi%
}
\usepackage[nameinlink,capitalise,noabbrev]{cleveref}
\newcommand{\htimes}{\overline{\wedge}}
\makeatletter
\newcommand{\colim@}[2]{%
  \vtop{\m@th\ialign{##\cr
    \hfil$#1\operator@font colim$\hfil\cr
    \noalign{\nointerlineskip\kern1.5\ex@}#2\cr
    \noalign{\nointerlineskip\kern-\ex@}\cr}}%
}
\newcommand{\colim}{%
  \mathop{\mathpalette\colim@{\rightarrowfill@\scriptscriptstyle}}\nmlimits@
}
\renewcommand{\varprojlim}{%
  \mathop{\mathpalette\varlim@{\leftarrowfill@\scriptscriptstyle}}\nmlimits@
}
\renewcommand{\varinjlim}{%
  \mathop{\mathpalette\varlim@{\rightarrowfill@\scriptscriptstyle}}\nmlimits@
}
\makeatother

\newtheorem{thm}{Theorem}[section]
\newtheorem{cor}[thm]{Corollary}
\newtheorem{prop}[thm]{Proposition}
\newtheorem{lem}[thm]{Lemma}
\newtheorem{conj}[thm]{Conjecture}
\newtheorem{quest}[thm]{Question}

\theoremstyle{definition}
\newtheorem{defn}[thm]{Definition}
\newtheorem{ex}[thm]{Example}

\theoremstyle{remark}
\newtheorem{rem}[thm]{Remark}

\newtheorem{nota}[thm]{Notation}

\numberwithin{equation}{section}
\newcommand{\lr}[1]{\langle{#1}\rangle}

\renewcommand{\frak}{\mathfrak}

\ExplSyntaxOn
\NewDocumentCommand{\createbunch}{ m O{} m }
 {
  \clist_map_inline:nn { #3 } { \cs_new_protected:cpn { #2 ##1 } { #1 { ##1 } } }
 }
\ExplSyntaxOff
\createbunch{\operatorname}{Sp,supp,Loc,im,Hom,Mod,Ind,Thick,CAlg,End,Fun,Spc,Supp,Spec,Sub,res,triv,Pro,FinGrp,FinGpd,Span,Mack,Fin,infl,SH,cosupp,Coloc,Tot,Ext,Pic,dual,ideal,Tel,thick}
\newcommand{\thickt}[1]{\thick_\otimes\langle #1 \rangle}

\newcommand{\unit}{\mathbf{1}}

\newcommand{\cal}{\mathcal}

\createbunch{\mathcal}[c]{C,H,E,A,D,T,F,G,M,P,K,Q,Z,V,N,S}
\createbunch{\mathbb}[b]{Q,D,Z,G}
\newcommand{\pics}{\mathfrak{pic}}
\title{The $\Sp_{k,n}$-local stable homotopy category}
\author{Drew Heard}
\DeclareMathOperator{\Ho}{Ho}

\Crefname{figure}{Figure}{Figures}
\Crefname{assu}{Assumption}{Assumptions}
\Crefname{lem}{Lemma}{Lemmas}
\Crefname{thm}{Theorem}{Theorems}
\Crefname{prop}{Proposition}{Propositions}
\begin{document}
\begin{abstract}Following a suggestion of Hovey and Strickland, we study the category of $K(k) \vee K(k+1) \vee \cdots \vee K(n)$-local spectra. When $k = 0$, this is equivalent to the category of $E(n)$-local spectra, while for $k = n$, this is the category of $K(n)$-local spectra, both of which have been studied in detail by Hovey and Strickland. Based on their ideas, we classify the localizing and colocalizing subcategories, and give characterizations of compact and dualizable objects. We construct an Adams type spectral sequence and show that when $p \gg n$ it collapses with a horizontal vanishing line above filtration degree $n^2+n-k$ at the $E_2$-page for the sphere spectrum. We then study the Picard group of $K(k) \vee K(k+1) \vee \cdots \vee K(n)$-local spectra, showing that this group is algebraic, in a suitable sense, when $p \gg n$. We also consider a version of Gross--Hopkins duality in this category. A key concept throughout is the use of descent. 
\end{abstract}
\maketitle

\tableofcontents
\section{Introduction}
In their memoir \cite{hovey_strickland99} Hovey and Strickland studied the categories of $K(n)$-local and $E_n$-local spectra in great detail. Here $K(n)$ is the $n$-th Morava $K$-theory; the spectrum whose homotopy groups are the graded field 
\[
K(n)_* \cong \mathbb{F}_{p}[v_n^{\pm 1}], \quad |v_n| = 2(p^n-1)
\]
and $E_n$ is the $n$-th Lubin--Tate spectrum, or Morava $E$-theory, with 
\[
(E_n)_* \cong W(\mathbb{F}_{p^n})[\![u_1,\ldots,u_{n-1}]\!][u^{\pm 1}], \quad |u_i| =0, |u|=2. 
\] 
As explained in the introduction of \cite{hovey_strickland99}, the Morava $K$-theories are the prime field objects in the stable homotopy category  - for a way to make that precise, see \cite{HopkinsSmith1998Nilpotence} or more specifically \cite[Corollary 9.5]{Balmer10b} - and are one of the fundamental objects in the chromatic approach to stable homotopy theory. 

 A deep result of Hopkins and Ravenel \cite{Ravenel1992Nilpotence} is that Bousfield localization with respect to $E_n$ is smashing, which simplifies the study of the category of $E_n$-local spectra considerably. On the other hand, localization with respect to $K(n)$ is not smashing \cite[Lemma 8.1]{hovey_strickland99}, and the monoidal unit $L_{K(n)}S^0 \in \Sp_{K(n)}$ is dualizable, but not compact. In the language of tensor-triangulated geometry, $\Sp_{K(n)}$ is a non-rigidly compactly generated category. Because of this, much of the work in \cite{hovey_strickland99} is therefore dedicated to understanding the more complicated category of $K(n)$-local spectra.

 By a Bousfield class argument, the category of $E_n$-local spectra is equivalent to the category of $K(0) \vee \ldots \vee K(n)$-local spectra. In this paper we study the categories of $K(k) \vee \ldots \vee K(n)$-local spectra for $0 \le k \le n$, which were suggested as `interesting to investigate' by Hovey and Strickland, see the remark after Corollary B.9 \cite{hovey_strickland99}. We write $L_{k,n}$ for the associated Bousfield localization functor. As we shall see, when $k \ne 0$, the category $\Sp_{k,n}$ of $K(k) \vee \ldots \vee K(n)$-local spectra behaves much like the category $\Sp_{K(n)} = \Sp_{n,n}$ of $K(n)$-local spectra. For example, it is an example of a non-rigidly compactly generated category; as soon as $k \ne 0$, the monoidal unit $L_{k,n}S^0 \in \Sp_{k,n}$ is dualizable, but not compact. However, the categories $\Sp_{k,n}$ for $k \ne n$ are in some sense more complicated than the case $k = n$; for example, $\Sp_{K(n)}$ has no non-trivial (co)localizing subcategories, while this is not true for $\Sp_{k,n}$ as long as $k \ne n$. 
\subsection{Contents of the paper}
We now describe the contents of the paper in more detail. We begin with a study of Bousfield classes, constructing some other spectra which are Bousfield equivalent to $K(k) \vee \cdots \vee K(n)$. In particular, we show that there is a Bousfield equivalence between a localized quotient of $BP$, denoted $E(n,J_k)$ and $K(k) \vee \cdots \vee K(n)$. For this reason, as well as brevity, we often say that $X$ is $E(n,J_k)$-local, instead of $K(k) \vee \cdots \vee K(n)$-local.

As was already noted by Hovey and Strickland \cite[Corollary B.9]{hovey_strickland99} $\Sp_{k,n}$ is an algebraic stable homotopy theory in the sense of \cite{HoveyPalmieriStrickl1997Axiomatic} with compact generator $L_{k,n}F(k)$, the localization of a finite spectrum of type $k$. We investigate some consequences of this; for example, analogous to Hovey and Strickland's formulas for $L_{K(n)}X$, in \Cref{prop:7.10}, we prove some formulas for $L_{k,n}X$ in terms of towers of finite type $k$ Moore spectra. Some of these results had previously been obtained by the author and Barthel and Valenzuela \cite{bhv1}. 

In \Cref{sec:tt-geometry} we investigate the tensor-triangulated geometry of $\Sp_{k,n}$. We begin by characterizing the compact objects in $\Sp_{k,n}$, culminating in \Cref{thm:compact_objects} which is a natural extension of Hovey and Strickland's results in the cases $k = 0,n$.  A classification of the thick ideals of $\Sp_{k,n}^{\omega}$ is an almost immediate consequence of this classification, see \Cref{thm:thick_subcategory} for the precise result. Of course, here we rely on the deep thick subcategory theorem in stable homotopy \cite{HopkinsSmith1998Nilpotence} and its consequences. Finally, we classify the localizing and colocalizing subcategories of $\Sp_{k,n}$ in \Cref{thm:loc_coloc_classifying}. We obtain the following. 
\begin{thm}
There is an order preserving bijection between (co)localizing subcategories of $\Sp_{k,n}$ and subsets of $\{k,\ldots,n\}$.
Moreover, the map that sends a localizing subcategory $\cC$ of $\Sp_{k,n}$ to its left orthogonal $\cC^{\perp}$ induces a bijection between the set of localizing and colocalizing subcategories of $\Sp_{k,n}$. The inverse map sends a colocalizing subcategory $\cal{U}$ to its right orthogonal ${}^{\perp}\cal{U}$. 
\end{thm}
We also compute the Bousfield lattice of $\Sp_{k,n}$ (\Cref{prop:bousfield_lattice}) and show that a form of the telescope conjecture holds (\Cref{thm:telescope}).

In \Cref{sec:descent_anss} we show that, as a consequence of the Hopkins--Ravenel smash product theorem, the commutative algebra object $E_n \in \Sp_{k,n}$ is descendable, in the sense of \cite{Mathew2016Galois}. This has a number of immediate consequences. For example, it implies the existence of a strongly convergent Adams type spectral sequence, which we call the $E(n,J_k)$-local $E_n$-Adams spectral sequence, computing $\pi_*(L_{k,n}X)$ for any spectrum $X$. Moreover, descendability implies this collapses with a horizontal vanishing line at a finite stage (independent of $X$). In the case of $X = S^0$ it is known that, in the cases $k = 0$ and $k= n$, this vanishing line already occurs on the $E_2$-page so long as $p \gg n$. In order to generalize this result, we first show that when $X = S^0$, the $E_2$-term  of the $E(n,J_k)$-local $E_n$-Adams spectral sequence spectral sequence can be given as the inverse limit of certain $\Ext$ groups computed in the category of $(E_n)_*E_n$-comodules, see \Cref{prop:limit_ext} for the precise result. We are then able to utilize a chromatic spectral sequence and Morava's change of rings theorem to show the following (\Cref{thm:vanishing}):
\begin{thm}\label{thm:vanishing_intro}
  Suppose $p-1$ does not divide $k+s$ for $0 \le s \le n-k$ (for example, if $p > n+1$), then in the $E_2$-term of the $E(n,J_k)$-local $E_n$-Adams spectral sequence converging to $L_{k,n}S^0$ we have $E_2^{s,t} = 0$ for $s > n^2+n-k$. 
\end{thm}
In the case $k = 0$, this recovers a result of Hovey and Sadofsky \cite[Theorem 5.1]{HoveySadofsky1999Invertible}.

As noted previously, so long as $k \ne 0$, the categories of dualizable and compact spectra do not coincide in $\Sp_{k,n}$; every compact spectrum is dualizable, but the converse does not hold, with the unit $L_{k,n}S^0$ being an example. In \Cref{sec:dualizable} we study the category of dualizable objects in $\Sp_{k,n}$. As a consequence of descendability, we show that $X \in \Sp_{k,n}$ is dualizable if and only if $L_{k,n}(E_n \wedge X)$ is dualizable in the category of $E(n,J_k)$-local $E_n$-modules. In turn, we show that this holds if and only if $L_{k,n}(E_n \wedge X)$ is dualizable (equivalently, compact) in the category of $E_n$-modules. We deduce that $X\in \Sp_{k,n}$ is dualizable if and only if its \emph{Morava module} $(E_{k,n})^{\vee}_*(X) \coloneqq \pi_*L_{k,n}(E_n \wedge X)$ is finitely-generated as an $(E_n)_*$-module, see \Cref{thm:dualspkn}. This generalizes a result of Hovey and Strickland, but even in this case our proof differs from theirs. 

It is an observation of Hopkins that the Picard group of invertible $K(n)$-local spectra is an interesting object to study \cite{hms_pic}. Likewise, Hovey and Sadofsky have studied the Picard group of $E(n)$-local spectra \cite{HoveySadofsky1999Invertible}. In \Cref{sec:picard} we study the Picard group $\Pic_{k,n}$ of $E(n,J_k)$-local spectra. Our first result, which is a consequence of descent, is that $X \in \Sp_{k,n}$ is invertible if and only if its Morava module $(E_{k,n})^{\vee}_*(X)$ is free of rank 1. We then study the Picard spectrum \cite{mathew_pic} of the category $\Sp_{k,n}$. Using descent again, we construct a spectral sequence whose abutment for $\pi_0$ is exactly $\Pic_{k,n}$. The existence of this spectral sequence in the case $k = n$ is folklore. We say that this spectral sequence is algebraic if the only non-zero terms in the spectral sequence occur in filtration degree 0 and 1. Using \Cref{thm:vanishing_intro} we deduce the following result (\Cref{thm:picskn}). In the case $k = n$, this is a theorem of Pstr\k{a}gowski \cite{1811.05415}.
\begin{thm}
    If $2p-2 \ge n^2+n-k$ and $p-1$ does not divide $k+s$ for $0 \le s \le n-k$, then $\Pic_{k,n}$ is algebraic. For example, this holds if $2p-2 > n^2+n$. 
\end{thm}
There is an interesting element in the $K(n)$-local Picard group, namely the Brown--Comenetz dual of the monochromatic sphere \cite[Theorem 10.2]{hovey_strickland99}. In \Cref{sec:bc_dual} we extend Brown--Comentz duality to the $E(n,J_k)$-local category. We do not know when the Brown--Comenetz dual of the monochromatic sphere defines an element of $\Pic_{k,n}$; this is not true when $k = 0$, and we provide a series of equivalent conditions for the general case in \Cref{prop:bc_equivalent}. 
\subsection*{The case $n = 2,k=1$}
The first example that has essentially not been studied in the literature is when $n = 2$ and $k = 1$, i.e., the category of $K(1) \vee K(2)$-local spectra. In \Cref{sec:balmer_dual} we give a computation of the Balmer spectrum of $K(1) \vee K(2)$-locall dualizable spectra. For this, we recall that Hovey and Strickland have conjectured a description of the Balmer spectrum $\Spc(\Sp_{K(n)}^{\dual})$ of dualizable objects in $K(n)$-local spectra \cite[p.~61]{hovey_strickland99}. This was investigated by the author, along with Barthel and Naumann, in \cite{bhn}. This admits a natural generalization to $\Sp_{k,n}^{\dual}$.  For $i \le n$, let $\cal D_{i}$ denote the category of $X \in \Sp_{k,n}^{\dual}$ such that $X$ is a retract of $Y \wedge Z$ for some $Y \in \Sp_{k,n}^{\dual}$ and some finite spectrum $Z$ of type at least $i$. We also set $\cD_{n+1} = (0)$. The conjecture is that these exhaust all the thick tensor-ideals of $\Sp_{k,n}^{\dual}$. We show in \Cref{thm:hovey_strickland_conjecture} that if this holds $K(n)$-locally (i.e., in $\Sp^{\dual}_{n,n}$), then it holds for all $\Sp_{k,n}^{\dual}$. In particular, since it is known to hold $K(2)$-locally by \cite[Theorem 4.15]{bhn} we obtain the following, see \Cref{cor:height_2_hovey_strickland}.
\begin{thm}
The Balmer spectrum of $K(1) \vee K(2)$-locally dualizable spectra $\Spc(\Sp_{1,2}^{\dual}) = \{ \cal{D}_1,\cal{D}_2,\cal{D}_3 \} $ with topology determined by $\overline{\{\cal{D}_j\}} = \{ \cal{D}_i \mid i \ge j \}$. In particular, if $\cC$ is a thick tensor-ideal of $\Sp_{1,2}^{\dual}$, then $\cC = \cal{D}_k$ for $0 \le k \le 3$. 
\end{thm}

 \subsection*{Conventions and notation}
 We let $\langle X \rangle$ denote the Bousfield class of a spectrum $X$.  The smallest thick tensor-ideal containing an object $A$ will be denoted by $\thickt{A}$ (it will always be clear in which category this thick subcategory should be taken in). Likewise, the smallest thick (respectively localizing) subcategory containing an object $A$ will be written as $\Thick(A)$ (respectively $\Loc(A)$). 
 \subsection*{Acknowledgments}
 It goes without saying that this paper owes a tremendous intellectual debt to Hovey and Strickland, in particular for the wonderful manuscript \cite{hovey_strickland99}. We also thank Neil Strickland for a helpful conversation, as well as his comments on a draft version of this document.  
\section{The category of \texorpdfstring{$\Sp_{k,n}$}{Sp{k,n}}-local spectra}\label{sec:bousfield}
\subsection{Chromatic spectra} We begin by introducing some of the main spectra that we will be interested in. 
\begin{defn}
  Let $BP$ denote the Brown--Peterson homotopy ring spectrum with coefficient ring 
  \[
BP_* \cong \bZ_{(p)}[v_1,v_1,\ldots]
  \]
  with $|v_i| = 2(p^i-1)$.
\end{defn}
\begin{rem}
  The classes $v_i$ are not intrinsically defined, and so the definition of $BP$ depends on a choice of sequence of generators; for example, they could be the Hazewinkel generators or the Araki generators. However, the ideals $I_n = (p,v_1,\ldots,v_{n-1})$ for $0 \le n \le \infty$ do not depend on this choice. 
\end{rem}

  By taking quotients and localizations of $BP$ (for example, using the theory of structured ring spectra \cite[Chapter V]{ekmm}), we can form new homotopy ring spectra. In particular, let $J_k$ denote a fixed invariant regular sequence $p^{i_0},v_1^{i_1},\ldots,v_{k-1}^{i_{k-1}}$ of length $k$. Then we can form the homotopy associative  ring spectrum $BPJ_k$ with 
  \[
(BPJ_k)_* \cong BP_*/J_k. 
  \]
  These were first studied by Johnson and Yosimura \cite{JohnsonlYosimura1980Torsion}. A detailed study on the product structure one obtains via this method can be found in \cite{Strickl1999Products}. 

\begin{defn}We let $E(n,J_k)$ for $n \ge k$ denote the Landweber exact spectrum with
\[
E(n,J_k)_* \cong v_n^{-1}((BPJ_k)_*/(v_{n+1},v_{n+2},\ldots)).
\]
Here Landweber exact means over $BPJ_k$ (as studied by Yosimura \cite{Yosimura1983operations}), that is, there is an isomorphism 
\[
E(n,J_k)_*(X) \cong (BPJ_k)_*(X) \otimes_{BP_*/J_k} E(n,J_k)_*.
\]
\end{defn}
\begin{ex}
  If $k=0$ (so that $J_k$ is the trivial sequence), then $E(n,J_0) \simeq E(n)$, Johnson--Wilson theory. For the other extreme, if $J_n = p,v_1,\ldots,v_{n-1}$, then $BPJ_n$ is the spectrum known as $P(n)$, and $E(n,J_n) \simeq K(n)$ is Morava $K$-theory \cite{hovey_strickland99}.
\end{ex}
\begin{defn}
  For $k \le n < \infty$, we let $\Sp_{k,n} \subseteq \Sp$ denote the full subcategory of $K(k) \vee K(k+1) \vee \cdots\vee K(n)$-local spectra.
\end{defn}
\begin{lem}
The inclusion $\Sp_{k,n} \hookrightarrow \Sp$ has a left adjoint adjoint $L_{k,n}$, and $\Sp_{k,n}$ is a presentable, stable $\infty$-category. 
\end{lem}
\begin{proof}
  This is a consequence of \cite[Proposition 5.5.4.15]{Lurie2009Higher}.
\end{proof}
\begin{rem}
  The category $\Sp_{k,n}$ and localization functor $L_{k,n}$ only depend on the Bousfield class $\langle K(k) \vee \cdots \vee K(n) \rangle$. 
\end{rem}
\begin{nota}
  We will follow standard conventions and write $\Sp_{0,n}$ as $\Sp_n$ and $\Sp_{n,n}$ as $\Sp_{K(n)}$. Similarly, the corresponding Bousfield localization functors will be denoted by $L_{0,n} = L_n$ and $L_{n,n} = L_{K(n)}$, respectively. 
\end{nota}
\begin{rem}\label{rem:lubin_tate}
  By \cite[Theorem 2.1]{Ravenel1984Localization} we have $\langle E(n) \rangle = \langle K(0) \vee \ldots \vee K(n) \rangle$. In fact, let $E$ be a $BP$-module spectrum that is Landweber exact over $BP$, and is $v_n$-periodic, in the sense that $v_n \in BP_*$ maps to a unit in $E_*/(p,v_1,\ldots,v_{n-1})$. Then Hovey has shown that $\langle E \rangle = \langle K(0) \vee \ldots \vee K(n) \rangle$ \cite[Corollary 1.12]{Hovey1995Bousfield}. In particular, this applies to the Lubin--Tate $E$-theory spectrum $E_n$ (see \cite{Rezk1998Notes}) with 
  \[
(E_n)_* \cong W(\mathbb{F}_{p^n})[\![u_1,\ldots,u_{n-1}]\!][u^{\pm 1}]
  \]
  or the completed version of $E$-theory used in \cite{hovey_strickland99} with
  \[
E_* \cong (E(n)_*)^{\wedge}_{I_n} \cong \bZ_{p}[v_1,\ldots,v_{n-1},v_n^{\pm 1}]^{\wedge}_{I_n}. 
  \]
\end{rem}
\subsection{Bousfield decomposition}
In the previous section we introduced the spectra $E(n,J_k)$ for $n \ge k$ and an invariant regular sequence $p^{i_0},\ldots,v_{k-1}^{i_{k-1}}$ of length $k$. We now give Bousfield decompositions for $E(n,J_k)$-local spectra. 
\begin{prop}\label{prop:bousfield_classes}
  There are equivalences of Bousfield classes
  \begin{enumerate}
    \item  (Johnson--Yosimura) $ \langle v_n^{-1}BPJ_k \rangle = \langle E(n,J_k) \rangle$,
    \item (Yosimura) $\langle E(n,J_k) \rangle = \bigvee_{i=k}^n \langle K(i) \rangle$.
    \item $\langle E_n/I_k \rangle = \bigvee_{i=k}^n \langle K(i) \rangle$.
  \end{enumerate}

\end{prop}
\begin{proof}
  Part (1) is \cite[Corollary 4.11]{JohnsonlYosimura1980Torsion}. (2) can be deduced from \cite{Yosimura1985Acyclicity} as we now explain. First, by \cite[Corollary 1.3 and Proposition 1.4]{Yosimura1985Acyclicity} along with (1), we have 
  \[
\langle E(n,J_k) \rangle = \langle v_n^{-1}BPJ_k \rangle   = \langle L_n BPJ_k \rangle  = \bigvee_{i=k}^n \langle v_i^{-1}P(i) \rangle
  \]
  By \cite[Corollary 1.8]{Yosimura1985Acyclicity} we have $\langle v_i^{-1}P(i) \rangle = \langle K(i) \rangle$, and hence (2) follows. For (3), we first note that by the thick subcategory theorem $\langle E_n/I_k \rangle = \langle E_n \wedge F(k) \rangle$ for some finite type $k$ spectrum. Since $E_n = \bigvee_{i=0}^n K(i)$, (3) then follows from the definition of a type $k$ spectrum. 
\end{proof}
\begin{rem}
  In other words, the category of $E(n,J_k)$-local spectra is equivalent to the category of $K(k)\vee \cdots K(n)$-local spectra. Note that this implies this category only depends on the length of the sequence, and not the integers $i_0,\ldots,i_{n-1}$. We will therefore sometimes say that a spectrum $X$ is $E(n,J_k)$-local if $X \in \Sp_{k,n}$. 
\end{rem}

\subsection{Algebraic stable homotopy categories}
We now begin by recalling the basics on algebraic stable homotopy theories, see \cite{HoveyPalmieriStrickl1997Axiomatic} in the triangulated setting. 
\begin{defn}
  A stable homotopy theory is a presentable, symmetric monoidal stable $\infty$-
category $(\cC, \otimes, \unit)$ where the tensor product commutes with all colimits.  It is algebraic if  there is a set $\cal{G}$ of compact objects such that the smallest localizing subcategory of $\cC$ containing all $G \in \cal{G}$ is $\cC$ itself. 
\end{defn}
\begin{rem}
  The assumptions on $\cC$ imply that it the functor $- \wedge Y$ has a right adjoint $F(Y,-)$, i.e., the symmetric monoidal structure on $\cC$ is closed. 
\end{rem}
\begin{rem}
  The associated homotopy category $\Ho(\cC)$ is then an algebraic stable homotopy theory in the sense of \cite{HoveyPalmieriStrickl1997Axiomatic}. We note that compactness can be checked at the level of the homotopy category, see \cite[Remark 1.4.4.3]{lurie-higher-algebra}. 
\end{rem}
  Applying \cite[Corollary B.9]{hovey_strickland99} and \cite[Theorem 3.5.1]{HoveyPalmieriStrickl1997Axiomatic} we have the following. 

\begin{prop}[Hovey--Strickland, Hovey--Palmieri--Strickland]\label{prop:hs_b.9}
  $\Sp_{k,n}$ is an algebraic stable homotopy category with compact generator $L_{k,n}F(k)$. The symmetric monoidal structure in $\Sp_{k,n}$ is given by
  \[
X \htimes Y \coloneqq L_{k,n}(X \wedge Y).
  \]
  Colimits are computed by taking the colimit in spectra and then applying $L_{k,n}$, while function objects and limimts are computed in the category of spectra. 
\end{prop}
\begin{rem}
  The most difficult part of the above proposition is that $L_{k,n}F(k)$ is a compact generator of $\Sp_{k,n}$. Indeed, one must show that the conditions of \cite[Proposition B.7]{hovey_strickland99} are satisfied and to do this, one at some point needs to invoke the thick subcategory theorem \cite{HopkinsSmith1998Nilpotence}, or one its consequences (such as the Hopkins--Ravenel smash product theorem \cite{Ravenel1992Nilpotence}). 
\end{rem}
\begin{rem}
  The localization $L_n = L_{0,n}$ is smashing (that is $L_nX \simeq L_nS^0 \wedge X$) by the Hopkins--Ravenel smash product theorem \cite{Ravenel1992Nilpotence} and in this case $X \htimes Y \simeq X \wedge Y$. However, if $k \ne 0$, then localization $L_{k,n}$ is not smashing as the following lemma shows, and so $X \htimes Y \not \simeq X \wedge Y$ in general. 
\end{rem}
\begin{lem}\label{lem:non_compact_unit}
  If $k \ne 0$, then $L_{k,n}$ is not smashing, and $L_{k,n}S^0$ is not compact in $\Sp_{k,n}$.
\end{lem}
\begin{proof}
  We first claim that $\langle L_{k,n}S^0 \rangle = \langle E(n) \rangle$. To see this, note that we have ring maps $L_nS^0 \to L_{k,n}S^0 \to L_{K(n)}S^0$, so that $\langle L_{K(n)}S^0 \rangle \le \langle L_{k,n}S^0 \rangle \le \langle L_nS^0 \rangle $. However, $\langle L_{K(n)}S^0 \rangle = \langle L_nS^0 \rangle = \langle E(n) \rangle$ \cite[Corollary 5.3]{hovey_strickland99}, so that these inequalities are actually equalities, and all three are Bousfield equivalent to $E(n)$.

  Suppose now that $L_{k,n}$ were smashing, so that $\langle L_{k,n}S^0 \rangle = \bigvee_{i=k}^n\langle  K(i) \rangle$ \cite[Proposition 1.27]{Ravenel1984Localization}.  Then, since $\langle E(n) \rangle \gneq \bigvee_{i=k}^n\langle  K(i) \rangle$ as soon as $k \ne 0$, we have obtained a contradiction.

  The second part is then a consequence of \cite[Theorem 3.5.2]{HoveyPalmieriStrickl1997Axiomatic}.
\end{proof}
\begin{rem}\label{rem:moore_spectra}
  Using the periodicity theorem of Hopkins and Smith \cite{HopkinsSmith1998Nilpotence}, Hovey and Strickland \cite[Section 4]{hovey_strickland99} constructed a sequence of ideals $\{I_j\}_{j} \subseteq \frak{m} \subseteq E_0$ and type $k$ spectra  $\{M_k(j)\}_{j}$ with the following properties (see also \cite[Remark 2.1]{barthel2018constructing}):
  \begin{enumerate}
    \item $I_{j+1} \subseteq I_j$ and $\cap_{j} I_j = 0$;
    \item $E_0/I_{j}$ is finite; and,
    \item $E_*(M_k(j)) \cong E_*/I_j$ and there are spectrum maps $q \colon M_k(j+1) \to M_k(j)$ realizing the quotient $E_*/M_k(j+1) \to E_*/M_k(j)$. 
  \end{enumerate}
  We call such a tower $\{M_k(j)\}_{j}$ a tower of generalized Moore spectrum of type $k$. 
\end{rem}
\begin{rem}\label{rem:limits}
  The tower as above is constructed in the homotopy category of spectra. However, as explained in \cite[Page 9]{HopkinsSmith1998Nilpotence} (see equation (15)), such sequential diagrams can always be lifted to a sequence of cofibrations between cofibrant objects, and in particular to a diagram in the $\infty$-category of spectra (the point is that such diagrams have no non-trivial homotopy coherence data). Then, the (co)limit in the $\infty$-categorical sense, agrees with the homotopy (co)limit used in Definition 2.2.3 and Definition 2.2.10 of \cite{HoveyPalmieriStrickl1997Axiomatic}. 
\end{rem}
\begin{nota}
  We write $M_{k,n}X$ for the fiber of the localization map $L_nX \to L_{k-1}X$. By definition, we set $M_{0,n} = L_n$. 
\end{nota}
\begin{lem}\label{lem:bousfield_mkn}
  We have an equality of Bousfield classes 
  \[
\langle M_{k,n}S^0 \rangle = \bigvee_{i=k}^n \langle K(i) \rangle. 
  \]
\end{lem}
\begin{proof}
Recall that, by definition, there is a cofiber sequence
\[
C_{k-1} S^0 \to S^0 \to L_{k-1}S^0. 
\]  
Applying $L_n$ to this and using $L_iL_j \simeq L_{\min(i,j)}$ we see that $M_{k,n}S^0 \simeq L_nC_{k-1}S^0 \simeq C_{k-1}L_nS^0$, where the last equivalence follows as both functors are smashing.  It follows from \cite[Proposition 5.3]{hovey_strickland99} that $\langle M_{k,n}S^0 \rangle = \bigvee_{i=k}^n \langle K(i) \rangle$ as claimed. 
\end{proof}
\begin{rem}
  In \cite[Proposition 7.10(e)]{hovey_strickland99} Hovey and Strickland give a formula for $L_{K(n)}X$ in terms of towers of generalized Moore spectra. We show now that their proof extends to $L_{k,n}X$. 
\end{rem}
\begin{prop}\label{prop:7.10}
  There are equivalences 
  \[
L_{k,n}X \simeq L_{F(k)}L_nX \simeq \varprojlim_{j} (L_nX \wedge M_k(j)) \simeq F(M_{k,n}S^0,L_nX). 
  \]
  where the limit is taken over a tower $\{M_k(j)\}$ of generalized Moore spectra of type $k$. 
\end{prop}
\begin{proof}
  We first note that $L_{k-1}X \simeq L_nL_{k-1}X \simeq L_nL_{k-1}^fX$, the latter by \cite[Corollary 6.10]{hovey_strickland99}. It follows that $M_{k,n}X \simeq L_nC_{k-1}^fX$, where $C_{k-1}^f$ is the acycliczation functor associated to $L_{k-1}^f$. By \cite[Proposition 7.10(a)]{hovey_strickland99} (and \Cref{rem:limits})  $\displaystyle C_{k-1}^fX \simeq \varinjlim_j D(M_k(j)) \wedge X$, so that $\displaystyle M_{k,n}X \simeq \varinjlim_{j} D(M_k(j)) \wedge L_nX$. It follows that 
  \[
\varprojlim_{j} (L_nX \wedge M_k(j)) \simeq F(M_{k,n}S^0,L_nX). 
  \]
  Moreover, by \cite[Proposition 7.10(a)]{hovey_strickland99} this is equivalent to $L_{F(k)}L_nX$.

  To finish the proof, we will show that $L_{k,n}X \simeq L_{F(k)}L_nX$. First, note that $X \to L_nX$ is an $L_nS^0$-equivalence, and $L_nX \to L_{F(k)}L_nX$ is an $F(k)$-equivalence, so that $X \to L_{F(k)}L_nX$ is an $L_nS^0 \wedge F(k)$-equivalence. But $L_nS^0 \wedge F(k) \simeq L_nF(k)$ and $\langle L_nF(k) \rangle = \bigvee_{i=k}^n \langle K(i) \rangle$ \cite[Proposition 5.3]{hovey_strickland99}. Therefore $X \to L_{F(k)}L_nX$ is a $K(k) \vee \ldots \vee K(n)$-equivalence, and we only need show that $L_{F(k)}L_nX$ is $K(k) \vee \ldots \vee K(n)$-local. But $L_{F(k)}L_nX \simeq F(M_{k,n}S^0,L_nX)$ and so it follows from \Cref{lem:bousfield_mkn} that $L_{F(k)}L_nX$ is $K(k) \vee \ldots \vee K(n)$-local. We conclude that $L_{k,n}X \simeq L_{F(k)}L_nX$, as required. 
\end{proof}
\begin{rem}
  The equivalence 
  \[
L_{k,n}X \simeq \varprojlim_{j} (L_nX \wedge M_k(j))
  \]
  has also been obtained in \cite[Proposition 6.21]{bhv1} using the theory and complete and torsion objects in a stable $\infty$-category. The next result is also contained in \cite[Corollary 6.17]{bhv1}. 
\end{rem}
\begin{prop}\label{prop:pullback}
  For any spectrum $X$ there is a pullback square
  \[
\begin{tikzcd}
L_nX \arrow[d] \arrow[r] & {L_{k,n}X} \arrow[d] \\
L_{k-1}X \arrow[r]       & {L_{k-1}L_{k,n}X.}          
\end{tikzcd}
  \]
\end{prop}
\begin{proof}
  This is a standard consequence of the Bousfield decomposition $\langle E(n) \rangle = \langle E(k-1) \rangle \vee \langle E(n,J_k)\rangle$ using, for example, \cite[Proposition 2.2]{bauer_bousfield_2014} (or in the $\infty$-categorical setting, see \cite{Antolin-CamarenaBarthel2022Chromatic}).  
\end{proof}
\begin{rem}
  Use \cite[Proposition 2.2]{bauer_bousfield_2014} one can deduce various other chromatic fracture squares. For example, we have a pullback square
    \[
\begin{tikzcd}
L_{k,n}X \arrow[d] \arrow[r] & {L_{k,h}X} \arrow[d] \\
L_{h+1,n}X \arrow[r]       & {L_{k,h}L_{h+1,n}X,}          
\end{tikzcd}
  \]
  for $k \le h \le n-1$. 
\end{rem}
\begin{rem}
  These types of iterated chromatic localizations have been investigated by Bellumat and Strickland \cite{strickland2019iterated}. Results such as the chromatic fracture square can be recovered from their work, however we do not investigate this in detail. 
\end{rem}
\begin{cor}\label{cor:moore_at_least_k}
  Suppose $M_j$ is a generalized Moore spectrum of type at least $k$, then $L_nM_j \simeq L_{k,n}M_j$. 
\end{cor}
\begin{proof}
By definition, $L_{k-1}M_j \simeq \ast$, and so by the pullback square of \Cref{prop:pullback} we must show that $L_{k-1}L_{k,n}M_j$ is contractible. Because $M_j$ is a finite complex, this is equivalent to $L_{k-1}((L_{k,n}S^0) \wedge M_j) \simeq \ast$, and the result follows.   
\end{proof}

\begin{defn}
  Let $\cal{M}_{k,n}$ denote the essential image of the functor $M_{k,n} \colon \Sp \to \Sp$. 
\end{defn}
\begin{thm}\label{thm:category_equivalence}
For any spectrum $X$ we have natural equivalences $M_{k,n}L_{k,n}X \simeq M_{k,n}X$ and $L_{k,n}X \simeq L_{k,n}M_{k,n}X$. It follows that there is an equivalence of categories $\cal{M}_{k,n} \simeq \Sp_{k,n}$ given by $L_{k,n}$, with inverse given by $M_{k,n}$. 
\end{thm}
\begin{proof}
  The proof of Hovey and Strickland in the case $k = n$ generalizes essentially without change. 

  By definition $M_{k,n}X$ fits into a cofiber sequence
  \[
M_{k,n}X \to L_nX \to L_{k-1}X,
  \]
  so applying $L_{k,n}$ gives a cofiber sequence
  \[
L_{k,n}M_{k,n}X \to L_{k,n}L_nX \to L_{k,n}L_{k-1}X. 
  \]
  But $\langle E(k-1) \rangle = \bigvee_{i=0}^{k-1} \langle K(k) \rangle$, so $L_{k,n}L_{k-1}X \simeq 0$, while clearly $L_{k,n}L_nX \simeq L_{k,n}X$. It follows that $L_{k,n}M_{k,n}X \simeq L_{k,n}X$. 

  Using \Cref{prop:7.10} we have $L_{k,n}X \simeq F(M_{k,n}S^0,L_nX)$, and so applying $F(-,L_nX)$ to the defining cofiber sequence for $M_{k,n}S^0$ we obtain a cofiber sequence
  \[
F(L_{k-1}S^0,L_nX) \to F(L_nS^0,L_nX) \to F(M_{k,n}S^0,L_nX)
  \] 
  or equivalently
  \[
F(L_{k-1}S^0,L_nX) \to L_nX  \to L_{k,n}X.
  \]
 It is easy to check that $F(L_{k-1}S^0,L_nX)$ is $E(k-1)$-local, and so by \Cref{lem:bousfield_mkn} we have $M_{k,n}F(L_{k-1}S^0,L_nX) \simeq \ast$. It follows that $M_{k,n}L_nX \simeq M_{k,n}X \simeq M_{k,n}L_{k,n}X$ as claimed. 
\end{proof}
\begin{rem}
  Once again, this result was obtained (by different methods) in \cite[Proposition 6.21]{bhv1}. 
\end{rem}
\section{Thick subcategories and (co)localizing subcategories}\label{sec:tt-geometry}
In this section we compute the thick subcategories of compact objects in $\Sp_{k,n}$ and (co)localizing subcategories of $\Sp_{k,n}$. When $k = 0$ or $k = n$ both results have been obtained by Hovey and Strickland. Along the way we give a classification of the compact objects in $\Sp_{k,n}$.

\subsection{Compact objects in $\Sp_{k,n}$}\label{sec:compact_objects}
In this section we characterize the compact objects in $\Sp_{k,n}$. We will use this in the next section to compute the thick subcategories of $\Sp_{k,n}^{\omega}$. 

We begin by recalling the notions of thick and (co)localizing subcategories. 
\begin{defn}
  Let $(\cC,\wedge,\unit)$ be an algebraic stable homotopy category, and let $\cD$ be a full, stable, subcategory.  
  \begin{enumerate}
    \item $\cD$ is called thick if it is closed under extensions and retracts. 
    \item $\cD$ is called localizing if it is thick and closed under arbitrary colimits. 
    \item $\cD$ is called colocalizing if it is thick and closed under arbitrary limits. 
    \item $\cD$ is a tensor-ideal if $X \in \cC$ and $Y \in \cD$ implies $X \wedge Y \in \cD$. 
    \item $\cD$ is a coideal if $X \in \cC$ and $Y \in \cD$ implies $F(X,Y) \in \cD$. 
  \end{enumerate}
  We will also speak of localizing (or thick) tensor-ideals and colocalizing coideals. 
\end{defn}
\begin{rem}
  In $\Sp_n$ the dualizable and compact objects coincide, and are precisely those that lie in the thick subcategory generated by the tensor unit $L_nS^0$. In categories whose tensor unit is not compact, such as $\Sp_{k,n}$ for $k \ne 0$, the dualizable and compact objects do not coincide - for example, the tensor unit is always dualizable, but is not compact (\Cref{lem:non_compact_unit}). In \cite{hovey_strickland99} Hovey and Strickland gave numerous characterizations of compact objects in $\Sp_{K(n)}$. In this section we extend some of these characterizations to $\Sp_{k,n}$. We first recall the concept of a nilpotent object in a symmetric monoidal category. 
\end{rem}
\begin{defn}\label{defn:nilpotence}
  We say that $X$ is $R$-nilpotent if $X$ lies in the thick $\otimes$-ideal generated by $R$, i.e., $X \in \Thick_{\otimes}\langle R \rangle$. 
\end{defn}
\begin{lem}
  The category of $E_n/I_k$-nilpotent spectra is the same in $\Sp_{k,n}, \Sp_n$ and $\Sp$. 
\end{lem}
\begin{proof}
  Using that $\langle E_n/I_k \rangle = \langle E(n,J_k) \rangle$ (\Cref{prop:bousfield_classes}) we see that $E_n/I_k \wedge X$ is always $E(n,J_k)$-local, from which the result easily follows. 
\end{proof}
\begin{rem}
In other words, we can talk unambiguously about the category of $E_n/I_k$-nilpotent spectra. 
\end{rem}
We will also need the following generalization of \cite[Lemma 6.15]{hovey_strickland99}. 
\begin{lem}\label{lem:hs_6.15}
  If $X$ is a finite spectrum of type at least $k$, then $L_{n}X \simeq L_{k,n}X$ is $E_n/I_k$-nilpotent. 
\end{lem}
\begin{proof}
  The argument is only a slightly adaptation of that given by Hovey and Strickland. By a thick subcategory argument, we can assume that $X = M_k$ is a generalized Moore spectrum of type $k$. By \cite{Ravenel1992Nilpotence} we have $L_nS^0 \in \thickt{E_n}$, and it follows that $L_nS^0 \wedge M_k \simeq L_nM_k \in \thickt{E_n \wedge M_k}$. But it is easy to see that $\thickt{E_n \wedge M_k} \simeq \thickt{E_n/I_k}$ and we are done.
\end{proof}
\begin{rem}
  The fact that $L_nS^0 \in \thickt{E_n}$ is equivalent to the claim that $E_n \in \Sp_n$ is \emph{descendable}, a condition we investigate further in \Cref{sec:descendability}. 
\end{rem}
  The compact objects in $\Sp_{k,n}$ can be characterized in the following ways, partially generalizing \cite[Theorem 8.5]{hovey_strickland99}. We note that every compact object in $\Sp_{k,n}$ is automatically dualizable by  \cite[Theorem 2.1.3]{HoveyPalmieriStrickl1997Axiomatic}; we investigate the dualizable objects in $\Sp_{k,n}$ in more detail in \Cref{sec:dualizable}. 

\begin{thm}\label{thm:compact_objects}
  The following are equivalent for $X \in \Sp_{k,n}\colon$
  \begin{enumerate}
    \item $X$ is compact.
    \item $X \in \thick\langle L_nF(k) \rangle.$
    \item $X$ is a retract of $L_nX' \simeq L_{k,n}X'$ for a finite spectrum $X'$ of type at least $k$.
    \item $X$ is a retract of $Y \wedge X'$ where $Y$ is dualizable and $X'$ is a finite spectrum of type at least $k$.
    \item $X$ is dualizable and $E_n/I_k$-nilpotent.
  \end{enumerate}
  The category $\Sp_{k,n}^{\omega} \subseteq \Sp_{k,n}^{\dual}$ is thick. Moreover, if $X \in \Sp_{k,n}^{\omega}$ and $Y \in \Sp_{k,n}^{\dual}$, then $X \wedge Y, F(X,Y)$ and $F(Y,X)$ lie in $\Sp_{k,n}^{\omega}$. In particular, $F(X,L_{k,n}S^0) \in \Sp_n^{\omega}$. 
\end{thm}
\begin{proof}
  The equivalence of (1) and (2) is \cite[Theorem 2.1.3]{HoveyPalmieriStrickl1997Axiomatic} along with \Cref{prop:hs_b.9}. (3) implies (2) because every finite spectrum of type at least $k$ lies in the thick subcategory generated by $F(k)$. The implication (1) implies (3) is the same as given by Hovey and Strickland \cite[Theorem 8.5]{hovey_strickland99}. Namely, suppose that $X \in \Sp_{k,n}$ is compact. By \Cref{prop:7.10} we have $X \simeq \varinjlim_{j} (X \wedge DM_k(j))$, so $[X,X] \simeq \varinjlim_j[X,X \wedge DM_k(j)]$. In particular, $X$ is a retract of $Y \coloneqq X \wedge DM_k(j)$. We claim that such a $Y$ is compact in $\Sp_n$. Indeed, let $\{ Z_i \}$ be a filtered diagram of of $E_n$-local spectra. Then, we have equivalences 
  \[
  \begin{split}
[Y,\varinjlim_i Z_i]_* &\simeq [X, L_nM_k(j) \wedge \varinjlim_i Z_i]_* \\
&\simeq [X, L_{k,n}\varinjlim_i(M_k(j) \wedge Z_i)]_* \\
&\simeq \varprojlim_i [X,M_k(j) \wedge Z_i]_* \\
&\simeq \varprojlim_i [Y,Z_i]_*.
\end{split}
  \]
  The first and last equivalence follow by adjunction, the second because $\langle L_nM_k(j) \rangle = \bigvee_{i=k}^n \langle K(i) \rangle$ (\cite[Proposition 5.3]{hovey_strickland99}) so that 
  \[L_nM_k(j) \wedge  \varinjlim_i Z_i \simeq L_nM_k(j) \wedge L_{k,n} \varinjlim_i Z_i \simeq L_{k,n}\varinjlim_{i}(M_k(j) \wedge Z_i) ,\]
while the third equivalence follows because $X \in \Sp_{k,n}^{\omega}$ by assumption and because $L_{k,n}\varinjlim_i$ is the colmit in $\Sp_{k,n}$. We have $K(i)_*Y = 0$ for $i<k$ and so \cite[Corollary 6.11]{hovey_strickland99} implies that $Y$, and hence $X$, is a retract of $L_{n}Z \simeq L_{k,n}Z$ for a finite spectrum $Z$ of type at least $k$. This shows that (1),(2) and (3) are equivalent.

Assume now that (4) holds. Note that $Y \wedge X'$ is $E(n,J_k)$-local, and moreover $Y \wedge X' \simeq Y \wedge L_{k,n}X'$, where $L_{k,n}X' \in \Sp_{k,n}^{\omega}$. By \cite[Theorem 2.1.3]{HoveyPalmieriStrickl1997Axiomatic} the smash product of a dualizable and compact object is compact, and so $X$ is a retract of a compact $E(n,J_k)$-local spectrum, and so is also compact, i.e., (1) holds. 

To see that (3) implies (5), we use a thick subcategory argument to reduce to the case that $X = L_nM_k \simeq L_{k,n}M_k$ is a localized generalized Moore spectrum of type $k$. Such an $X$ is clearly dualizable and is additionally $E_n/I_k$-nilpotent by \Cref{lem:hs_6.15}. 

  Now suppose that $X$ satisfies (5). Following Hovey and Strickland \cite[Proof of Corollary 12.16]{hovey_strickland99} let $\cal{J}$ be the collection of those spectra $Z \in \Sp_{k,n}$ such that $Z$ is a module over a generalized Moore spectrum of type $i$ (for a fixed $i$, $k \le i \le n$). By \cite[Proposition 4.17]{hovey_strickland99} $\cal{J}$ forms an ideal. Because $K(i) \wedge Z$ is non-zero and a wedge of suspensions of $K(i)$, $\cal{J}$ contains the ideal of $K(i)$-nilpotent spectra. Moreover, it follows from the Bousfield decomposition $\langle E_n/I_k \rangle = \bigvee_{i=k}^n \langle K(i) \rangle$ that $K(i) \wedge E_n/I_k \neq 0$, and so $\thick_{\otimes}\langle K(i) \rangle \subseteq \thick_{\otimes} \langle E_n/I_k \rangle$, i.e., every $K(i)$-nilpotent spectrum (for $k \le i \le n$) is also $E_n/I_k$-nilpotent.   In particular, $X \in \cal{J}$, so that $X$ is retract of a spectrum of the form $Y \wedge X$ where $Y$ is a generalized Moore spectrum of type $i$, and so (4) holds.  

  Finally, we prove the subsidiary claims. It is immediate from (2) that $\Sp_{k,n}^{\omega} \subseteq \Sp_{k,n}^{\dual}$ is thick, and it is an ideal by \cite[Theorem 2.1.3(a)]{HoveyPalmieriStrickl1997Axiomatic}. Because generalized Moore spectra are self-dual (see \cite[Proposition 4.18]{hovey_strickland99}), (c) implies that $\Sp_{k,n}^{\omega}$ is closed under Spanier--Whitehead duality. Therefore, $F(X,Y) \simeq F(X,L_{k,n}S^0) \htimes Y$ and $F(Y,X) \simeq X \htimes F(Y,L_{k,n}S^0)$ lie in $\Sp_{k,n}^{\omega}$.  
\end{proof}
\begin{rem}
  When $k = 0$, then $X$ is compact if and only if $X$ is dualizable \cite[Theorem 6.2]{hovey_strickland99}. To reconcile this with (5) of the previous theorem, we note that every spectrum $X \in \Sp_n$ is $E_n/I_0 \simeq E_n$-nilpotent \cite[Theorem 5.3]{Ravenel1992Nilpotence}. 
  \end{rem}
\subsection{The thick subcategory theorem}
We now give a thick subcategory theorem for $\Sp_{k,n}^{\omega}$. As we shall see, given \Cref{thm:compact_objects} this is an immediate consequence of the classification of thick subcategories of $\Sp_{n}^{\omega}$, which ultimately relies on the Devinatz--Hopkins--Smith nilpotence theorem. 
\begin{defn}
   For $0 \le j \le n+1$ let $\cC_j$ denote the thick subcategory of $\Sp_{n}$ consisting of all compact spectra $X$ such that $K(i)_*X = 0$ for all $i < j$, i.e., 
   \[
\cC_j = \{ X \in \Sp_n^{\omega} \mid K(i)_*X = 0 \text{ for all } i < j\}.
   \]
 \end{defn} 
 \begin{rem}
   By \cite[Proposition 6.8]{hovey_strickland99} we equivalently have 
   \[
\cC_j = \{X \in \Sp_n^{\omega} \mid K(j-1)_*X = 0\}.
   \]
 \end{rem}
 \begin{rem}
    We have 
 \[
\cC_0 \supsetneq \cC_1 \supsetneq \cdots \supsetneq \cC_{n+1} = (0),
 \]
 and moreover $L_nF(j)$ is in $\cC_j$, but not $\cC_{j+1}$. 

 We now present the result of Hovey and Strickland \cite[Theorem 6.9]{hovey_strickland99}. 
 \end{rem}
\begin{thm}[Hovey--Strickland]\label{thm:hs_thick}
 If $\cC$ is a thick subcategory of $\Sp_n^{\omega}$, then $\cC = \cC_j$ for some $j$ such that $0 \le j \le n+1$. 
\end{thm}
\begin{rem}
This result can be restated in terms of the Balmer spectrum of $\Sp_n^{\omega}$ \cite{balmer_prime}. In particular, we have
\[
\Spc(\Sp_n^{\omega}) \cong \{\cC_1,\ldots,\cC_{n+1} \}
\]
with topology determined by the closure operator $\overline {\{ \cC_j \}} = \{ \cC_i | i \ge j\}$. This is in fact equivalent to \Cref{thm:hs_thick}, essentially by the same argument as in \cite[Proposition 3.5]{bhn}. 

The Balmer support of $X \in \Sp_n^{\omega}$, defined in \cite[Definition 2.1]{balmer_prime}, is given by 
\[
\supp(X) = \{ \cP \in \Spc(\Sp_n^{\omega}) \mid X \not \in \cP\}.
\]
Note that $X \not \in \cC_j \iff K(j-1)_*X \ne 0$. Therefore, by \Cref{thm:hs_thick} we have
\[
\begin{split}
\supp(X) &= \{ i \in \{ 0,\ldots, n\} \mid K(i)_*X \ne 0 \} \\
&= \{ i \in \{ 0,\ldots, n\} \mid K(i) \wedge X \ne 0 \}.
\end{split}
\]
For a thick subcategory $\mathcal{J}$, we define $\supp(\mathcal J) = \bigcup_{X \in \mathcal J}\supp(X)$. Then, Balmer's classification result \cite[Theorem 4.10]{balmer_prime} shows that there is a bijection
 \[
\{ \text{thick subcategories of $\Sp_n^{\omega}$}\} \xrightarrow[\supp]{\sim} \{\text{specialization closed subsets of \{$0,\ldots,n$\}}\}.
 \]
with the topology on \{$0,\ldots,n$\} determined by $\overline{\{ k \}} = \{ k,k+1,\ldots,n \}$, with inverse given by sending a specialization closed subset $Y$ to $\{ X \in \Sp_n^{\omega}\mid \supp(X) \subseteq  Y\} $.  Note that there are exactly $n+2$ such specialization closed subsets, namely $ \emptyset $ and the subsets $\{ k,\ldots, n\}$ for $k=0,\ldots,n$. The thick subcategory $\cal{C}_{n+1}$ corresponds to $\emptyset$ under this bijection, while $\cal{C}_k$ corresponds to $\{k,\ldots, n\}$ for $0 \le k \le n$. 
\end{rem}

Given the classification of compact $E(n,J_k)$-local spectra in \Cref{thm:compact_objects}, we deduce the following.
\begin{lem}\label{lem:compact_thick}
The category of compact $E(n,J_k)$-local spectra, $\Sp_{k,n}^{\omega}$, is equivalent to the thick subcategory $\cC_k \subseteq \Sp_n^{\omega}$.
\end{lem}
\begin{proof}
  By \cite[Corollary 6.11]{hovey_strickland99} if $X \in \cal{C}_k$, then $X$ is a retract of $L_nY \simeq L_{k,n}Y$ for some finite spectrum $Y$ of type of least $k$. Then $X$ is a compact $E(n,J_k)$-local spectrum by \Cref{thm:compact_objects}. Conversely, if $X$ is a compact $E(n,J_k)$-local spectrum, then $X$ is a retract of $L_nY \simeq L_{k,n}Y$ for $Y$ a finite spectrum $Y$ of type of least $k$, again by \Cref{thm:compact_objects}. Therefore $K(i)_*X = 0$ for $i < k$ and $X \in \mathcal{C}_k$. 
\end{proof}
\begin{thm}[Thick subcategory theorem]\label{thm:thick_subcategory}
There is a bijection
\[
\{\text{thick subcategories of $\Sp_{k,n}^{\omega}$}\} \xrightarrow[\supp]{\sim} \{\text{specialization closed subsets of \{$k,\ldots,n$\}}\}.
 \]
 with inverse given by sending a specialization closed subset $Y$ to $\{ X \in \Sp_{k,n}^{\omega}\mid \supp(X) \subseteq  Y\} $. In particular, if $\cC$ is a thick subcategory of $\Sp_{k,n}^{\omega}$, then $\cC = \cC_j$ for some $j$ such that $k \le j \le n+1$.  
\end{thm}

\begin{proof}
  This follows by combining \Cref{thm:hs_thick,lem:compact_thick}. 
\end{proof}
\begin{rem}
  Note that $\Sp_{k,n}^{\omega}$ is not a tensor-triangulated category when $k \ne 0$, as it does not have a tensor unit. Therefore, we cannot speak of the Balmer spectrum of $\Sp_{k,n}^{\omega}$. 
\end{rem}

We also have a nilpotence theorem.
\begin{prop}
  Let $X \in \Sp_{k,n}^{\omega}$, and $u \colon \Sigma^dX \to X$ a self-map such that $K(i)_*u$ is nilpotent for $k \le i \le n$. Then, $u$ is nilpotent, i.e., the $j$-fold composite $u \circ \cdots \circ u \colon \Sigma^{jd}X \to X$ is trivial for large enough $j$.  
\end{prop}
\begin{proof}
  In light of \Cref{lem:compact_thick}, this follows from \cite[Corollary 6.6]{hovey_strickland99}. 
\end{proof}

\subsection{Localizing and colocalizing subcategories}
In this section we calculate the (co)localizing (co)ideals  of $\Sp_{k,n}$. We first observe that every (co)localizing subcategory is automatically a (co)ideal, so it suffice in fact to concentrated on (co)localizing subcategories.
\begin{lem}\label{lem:auto_ideal}
  Every (co)localizing subcategory of $\Sp_{k,n}$ is a (co)ideal.
\end{lem}
\begin{proof}
We prove the case of localizing subcategories - the case of colocalizing subcategories is similar.\footnote{We thank Neil Strickland for providing this argument, which simplifies a previous argument.} To that end, let $\cal C \subseteq \Sp_{k,n}$ be a localizing subcategory, and consider the collection $\cal D = \{ X \in \Sp \mid X \htimes \cal C \subseteq \cal C \}$. This is a localizing subcategory of $\Sp$ containing $\mathbb{S}$, and hence $\cal D = \Sp$ itself. It follows that $\cal{C}$ is a localizing ideal. 
\end{proof}
\begin{rem}
  We remind the reader that $\unit$ is not compact in $\Sp_{k,n}$ unless $k = 0$ (see \Cref{lem:non_compact_unit}). Therefore, in all other cases, $\unit$ is a non-compact generator of $\Sp_{k,n}$. 
\end{rem}
\begin{nota}
  Throughout this section we let $\cal{Q} = \{ k,\ldots,n \}$. 
\end{nota}
We begin by defining a notion of support and cosupport in $\Sp_{k,n}$, extending the notion of support defined previously for $\Sp_n^{\omega}$. 
\begin{defn}
    For a spectrum $X \in \Sp_{k,n}$ we define the support and cosupport of $X$ by
  \[
\begin{split}
  \supp(X) &= \{ i \in \cal{Q} \mid K(i) \wedge X \ne 0 \} \\
\cosupp(X) &= \{ i \in \cal{Q} \mid F(K(i),X) \ne 0  \} . 
\end{split}
  \]
\end{defn}
\begin{ex}\label{ex:supp_kn}
Because $K(i) \wedge K(j) = 0$ if $i \ne j$, and $K(i) \wedge K(i) \ne 0$ \cite[Theorem 2.1]{Ravenel1984Localization}, we have
\[
\supp(K(i)) = i
\]
for $i \in \cal{Q}$. On the other hand, $K(i)^*K(j) = \Hom_{K(i)_*}(K(i)_*K(j),K(i)_*)$, and so
\[
\cosupp(K(i)) = i
\]
as well. 
\end{ex}
  \begin{rem}
   The notion of support is slightly ambiguous, as objects can live in multiple categories. For example $L_{K(n)}S^0 \in \Sp_{i,n}$ for all $0 \le i \le n$, and in fact has different support in each category. However, it should also be clear in which category we are considering the support. 
  \end{rem}

\begin{rem}\label{rem:support_on_p}
  Because $K(i) \wedge X$ is always $K(i)$-local, we equivalently have 
  \[
 \supp(X) = \{ i \in \cal{Q} \mid K(i) \htimes X \ne 0 \}.
  \]
\end{rem}

\begin{rem}\label{rem:hovey_strickland_support}
  In \cite[Definition 6.7]{hovey_strickland99} Hovey and Strickland define the support of an object by
  \[
\supp_{HS}(X) = \{ i \mid K(i) \wedge X \ne 0\}.
  \]
  By definition then, $\supp(X) = \supp_{HS}(X) \cap \cal{Q}$. 
\end{rem}
Support and cosupport are well behaved with respect to products and function objects in $\Sp_{k,n}$. 
\begin{lem}\label{lem:supp_cosupp_tensor_hom}
  For any $X,Y \in \Sp_{k,n}$ there are equalities 
  \[
\supp(X \htimes Y) = \supp(X) \cap \supp(Y). 
  \]
  \[
\cosupp(F(X,Y)) = \supp(X) \cap \cosupp(Y). 
  \]
\end{lem}
\begin{proof}
  Because $K(i) \wedge X$ is always $K(i)$-local, we have $K(i) \wedge (X \htimes Y) \simeq K(i) \wedge X \wedge Y$, and it is clear that $\supp(X \htimes Y) \subseteq \supp(X) \cap \supp(Y)$. The converse follows because $K(i)_*$ is a graded field; if $K(i) 
  \wedge X \wedge Y \simeq \ast$ then either $K(i) \wedge X \simeq \ast$ or $K(i) \wedge Y \simeq \ast$. 

  For the statement about cosupport, suppose $i \in \cosupp(F(X,Y))$, i.e., $F(K(i),F(X,Y)) \ne 0$. By adjunction we must have $F(K(i) \wedge X,Y) \ne 0$ as well as $F(X,F(K(i),Y)) \ne 0$, so that $K(i) \wedge X \ne 0$ and $F(K(i),Y) \ne 0$. This shows that $\cosupp(F(X,Y)) \subseteq \supp(X) \cap \cosupp(Y)$. For the converse, let $i \in \supp(X) \cap \cosupp(Y)$, and consider $F(K(i),F(X,Y)) \simeq F(K(i) \wedge X,Y)$. Because $i \in \supp(X)$, and $K(i)_*$ is a graded field, $K(i) \wedge X$ is a wedge of suspensions of $K(i)$, and it suffices to show that $F(K(i),Y) \not \simeq 0$, which is precisely the statement that $i \in \cosupp(Y)$. Therefore, $i \in \cosupp(F(X,Y))$, as required. 
\end{proof}
\begin{nota}
  For an arbitrary collection $\cal{C}$ of objects we set 
  \[
  \begin{split}
  \supp(\cC) &= \bigcup_{X \in \cC} \supp(X) \\
\cosupp(\cC) &= \bigcup_{X \in \cC} \cosupp(X) 
\end{split}
  \]

  For a subset $\cT \subseteq \cQ$ we also define
  \[
\begin{split}
  \supp^{-1}(\cT) &=  \{ X \in \Sp_{k,n} \mid \supp(X) \subseteq \cT \}\\
  \cosupp^{-1}(\cT) &= \{ X \in \Sp_{k,n} \mid \cosupp(X) \subseteq \cT \}
\end{split}
  \]
\end{nota}
  \begin{lem}
    For a subset $\cT \subseteq \cQ$, $\supp^{-1}(\cT)$ and $\cosupp^{-1}(\cT)$ are localizing and colocalizing subcategories of $\Sp_{k,n}$ respectively. 
  \end{lem}
  \begin{proof}
    We simply note that \[
    \supp^{-1}(\cT) = \{ X \in \Sp_{k,n} \mid K(i) \wedge X = 0 \text{ for all } i  \in \cQ\setminus \cT \} \] 
    and 
    \[
\cosupp^{-1}(\cT) =  \{X \in \Sp_{k,n} \mid F(K(i),X) \simeq 0 \text{ for all } i \in \cQ\setminus \cT  \},
    \]
   which are clearly (co)localizing subcategories of $\Sp_{k,n}$. 
  \end{proof}
  We thus obtain maps
  \begin{equation}
\begin{Bmatrix}\label{eq:maps_loca}
\text{Localizing subcategories} \\
\text{ of $\Sp_{k,n}$}
\end{Bmatrix} 
\xymatrix{ \ar@<0.5ex>[r]^{\supp} & \ar@<0.5ex>[l]^{\supp^{-1}}}
\begin{Bmatrix}
\text{subsets of $\cQ$}  
\end{Bmatrix}
\end{equation}
and 
  \begin{equation}\label{eq:maps_coloca}
\begin{Bmatrix}
\text{Colocalizing subcategories} \\
\text{ of $\Sp_{k,n}$}
\end{Bmatrix} 
\xymatrix{ \ar@<0.5ex>[r]^{\cosupp} & \ar@<0.5ex>[l]^{\cosupp^{-1}}}
\begin{Bmatrix}
\text{subsets of $\cQ$}  
\end{Bmatrix}
\end{equation}

We will see that these are bijections.  We need the following local-global principle, which is a slight variant of that given by Hovey and Strickland \cite[Proposition 6.18]{hovey_strickland99}.

\begin{prop}[Local-global principle]\label{prop:lg}
  For any $X \in \Sp_{k,n}$ we have 
  \[
X \in \Loc_{\Sp_{k,n}}(X) = \Loc_{\Sp_{k,n}}(K(i) \mid i \in \supp(X))
  \]
  and
  \[
X \in \Coloc_{\Sp_{k,n}}(X) = \Coloc_{\Sp_{k,n}}(K(i) \mid i \in \cosupp(X)). 
  \]
\end{prop}
\begin{proof}
Because $X \in \Sp_{k,n} \subseteq \Sp_n$ applying \cite[Proposition 6.18]{hovey_strickland99} we have 
\begin{equation}\label{eq:hs_loc_global}
X \in \Loc_{\Sp}(X) = \Loc_{\Sp}(K(i) \mid i \in \supp_{HS}(X))
\end{equation}
and
\[
X \in \Coloc_{\Sp}(X) = \Coloc_{\Sp}(K(i) \mid i \in \cosupp(X))
\]
The result for colocalizing subcategories is then clear, as we get the same result taking the colocalizing subcategories in $\Sp_{k,n}$. For localizing subcategories we apply \cite[Lemma 2.5]{bchv} to \eqref{eq:hs_loc_global} with the colimit preserving functor $F = L_{k,n} \colon \Sp \to \Sp_{k,n}$ to see that
  \[
L_{k,n}X \simeq X \in \Loc_{\Sp_{k,n}}(X) = \Loc_{\Sp_{k,n}}(K(i) \mid i \in \supp_{HS}(X) \cap \cQ ),
  \]
  where we have used that 
\[
  L_{k,n}K(i) \simeq \begin{cases}
    K(i) & \text{ if } i \in \cQ \\
    0 &\text{ if } i \not \in \cQ.
  \end{cases}
\]
   As previously noted (\Cref{rem:hovey_strickland_support}) $\supp_{HS}(X) \cap \cQ =\supp(X)$, and the result follows. 
\end{proof}
\begin{rem}\label{rem:detection}
 If follows from the local-global principle that both support and cosupport detect trivial objects:
 $$\supp(X) = \emptyset \iff X  \simeq 0 \iff \cosupp(X) = \emptyset.$$  
\end{rem}
\begin{cor}\label{cor:suppinverse}
  We have 
  \[
\supp^{-1}(\cT) = \Loc_{\Sp_{k,n}}(K(i) \mid i \in \cT). 
  \]
  and
  \[
\cosupp^{-1}(\cT) = \Coloc_{\Sp_{k,n}}(K(i) \mid i \in \cT). 
  \]
\end{cor}
\begin{proof}
  Let $\cA = \Loc_{\Sp_{k,n}}(K(i) \mid i \in \cT)$. Because $\supp(K(i)) = i$ (\Cref{ex:supp_kn}), it is clear that $\cA \subseteq \supp^{-1}(\cT)$. Conversely, if $X \in \supp^{-1}(\cT)$, then \Cref{prop:lg} shows that 
  \[
X \in \Loc_{\Sp_{k,n}}(K(i) \mid i \in \cT) = \cA,
  \]
  so $\supp^{-1}(\cT) = \cA$, as claimed. The argument for colocalizing categories is similar. 
\end{proof}
We now give the promised classification of localizing and colocalizing subcategories. 
\begin{thm} \hfill\label{thm:loc_coloc_classifying}
\begin{enumerate}
  \item The maps \eqref{eq:maps_loca} give an order preserving bijection between localizing subcategories of $\Sp_{k,n}$ and subsets of $\cQ = \{ k,\ldots, n\}$. 
  \item The maps \eqref{eq:maps_coloca} give an order preserving bijection between colocalizing subcategories of $\Sp_{k,n}$ and subsets of $\cQ = \{ k,\ldots, n\}$. 
\end{enumerate}
\begin{proof}
Let $\cC \subseteq \Sp_{k,n}$ be a localizing subcategory and $\cT \subseteq \{ k,\ldots,n \}$ a subset. Then via \Cref{cor:suppinverse} and basic properties of support
\[
\supp(\supp^{-1}(\cT)) = \bigcup_{i \in \cT}\supp(K(i)) = \cT. 
\]

Now suppose that $X \in \cC$, so that $\supp(X) \subseteq \supp(\cC)$. It follows from the definitions that $X \in \supp^{-1}(\supp(\cC))$, and so $\cC \subseteq \supp^{-1}(\supp(\cC))$. We are therefore reduced to showing that $\supp^{-1}(\supp(\cC)) \subseteq \cC$. To that end, let $Y \in \supp^{-1}(\supp(\cC))$, so that $\supp(Y) \subseteq \supp(C)$. Using the local-global principle, \Cref{prop:lg}, we then have
    \[
    \begin{split}
Y &\in \Loc_{\Sp_{k,n}}(K(i) \mid i \in \supp(Y))\\
& \subseteq \Loc_{\Sp_{k,n}}(K(i) \mid i \in \supp(\cC)) \\
& = \cC,
\end{split}
    \]
    where the last equality follows from \Cref{prop:lg} again. The proof for colocalizing subcategories is analogous. 
\end{proof}  
\begin{nota}
  For the following, we recall that for $\cC \subseteq \Sp_{k,n}$ the right orthogonal $\cC^{\perp}$ is defined as 
  \[
\cC^{\perp} = \{ Y \in \Sp_{k,n} \mid  F(X,Y) = 0 \text{ for all } X \in \cC\}.
  \]
  Similarly, the left orthogonal ${}^\perp \cC$ is 
   \[
{}^\perp \cC = \{ Y \in \Sp_{k,n} \mid  F(Y,X) = 0 \text{ for all } X \in \cC\}.
  \]
  Moreover, the right orthogonal is a colocalizing subcategory, and the left orthogonal is a localizing subcategory. 
\end{nota}
\end{thm}

\begin{cor}
  The map that sends a localizing subcategory $\cC$ of $\Sp_{k,n}$ to $\cC^{\perp}$ induces a bijection
  \begin{equation}
  \begin{Bmatrix}
\text{Localizing subcategories} \\
\text{ of $\Sp_{k,n}$}
\end{Bmatrix} 
\xymatrix{ \ar[r]^{\sim} &}
\begin{Bmatrix}
\text{Colocalizing subcategories} \\
\text{ of $\Sp_{k,n}$} 
\end{Bmatrix}
\end{equation}
 The inverse map sends a colocalizing subcategory $\cal{U}$ to ${}^{\perp}\cal{U}$. 
\end{cor}
\begin{proof}
 We follow \cite[Corollary 9.9]{BensonIyengarKrause2012Colocalizing}. Let $\cC$ be a localizing subcategory, then using \Cref{rem:detection,lem:supp_cosupp_tensor_hom}
  \[
\begin{split}
  \cC^{\perp} &= \{ Y \in \Sp_{k,n} \mid F(X,Y) = 0 \text{ for all } X \in \cC \} \\
  &= \{ Y \in \Sp_{k,n} \mid \cosupp(Y) \cap \supp(\cC) = \emptyset \}\\
  &= \{ Y \in \Sp_{k,n} \mid \cosupp(Y) \subseteq \cQ \setminus \supp(\cC) \}\\
  &= \cosupp^{-1}(\cQ \setminus \supp(\cC)) .
  \end{split}
  \]
   Similarly, if $\cal{U}$ is a colocalizing subcategory, then 
  \[
\begin{split}
  {}^{\perp}\cal{U} &= \{ X \in \Sp_{k,n} \mid F(X,Y) = 0 \text{ for all } Y \in \cal{U} \} \\
  &= \{ X \in \Sp_{k,n} \mid \cosupp(\cal{U}) \cap \supp(X) = \emptyset \}\\
  &= \{ X \in \Sp_{k,n} \mid \supp(X) \subseteq \cQ \setminus \cosupp(\cal{U}) \}\\
  &= \supp^{-1}(\cQ \setminus \cosupp(\cal{U})) .
  \end{split}
  \]
  It follows that under the equivalences of \Cref{thm:loc_coloc_classifying}, the maps $\cC \mapsto \cC^{\perp}$ and $\cal{U} \mapsto {}^{\perp}\cal{U}$ correspond to the map $\cal{Q} \to \cal{Q}$ sending a subset to its complement, and are thus mutually inverse bijections.  
\end{proof}

\subsection{The Bousfield lattice}
We recall the basics on the Bousfield lattice of an algebraic stable homotopy theory. In order to avoid confusion with the (localized) categories of spectra considered previously we let $(\cC,\wedge,\unit)$ denote a tensor triangulated category.
\begin{defn}
  The Bousfield class of an object $X \in \cC$ is the full subcategory of objects
  \[
\langle X \rangle = \{ W \in \cC \mid X \wedge W = 0 \}
  \]
\end{defn}
\begin{rem}
   We always assume that our categories are compactly-generated and hence there is a set of Bousfield classes \cite[Theorem 3.1]{IyengarKrause2013Bousfield}.
\end{rem}
\begin{rem}
  We let $\BL(\cC)$ denote the set of Bousfield classes of $\cC$. As is known, this has a lattice structure, which we now describe. We say that $\lr{X} \le \lr{Y}$ if $Y \wedge W =0 \implies X \wedge W =0 $. Hence, $\langle 0 \rangle$ is the minimum Bousfield class, and $\lr{\unit}$ is the maximum. The join is defined by $\bigvee_{i \in I}\lr{X_i} = \lr{\coprod_{i \in I} X_i}$, and the meet is the join of all lower bounds.
\end{rem}

\begin{prop}\label{prop:bousfield_lattice}
  The Bousfield lattice $\BL(\Sp_{k,n})$ is isomorphic to the lattice of subsets of $\cQ$ via the map sending $\langle X \rangle$ to $\supp(X)$. 
\end{prop}
\begin{proof}
 Define a map that sends $\cT \subseteq \cQ$ to $\langle \bigvee_{i \in \cT} K(i) \rangle$ in $\BL(\Sp_{k,n})$. We claim that this gives the necessary inverse map. By the local-global principle (\Cref{prop:lg}) we have
 \[
\Loc_{\Sp_{k,n}}(X) = \Loc_{\Sp_{k,n}}(K(i) \mid i \in \supp(X)).
 \]
 In particular, $X \htimes W \simeq 0$ if and only if $K(i) \htimes W \simeq 0$ for all $i \in \supp(X)$, so that 
 \begin{equation}\label{eq:bousfield_decomp}
\langle X \rangle = \langle \bigvee_{i \in \supp(X)}K(i) \rangle.
 \end{equation}
 The result then follows by direct computation. 
\end{proof}
\subsection{The telescope conjecture and variants}
We begin by considering variants of the telescope conjecture in the localized categories $\Sp_{k,n}$ using work of Wolcott \cite{Wolcott2015Variations}. 
\begin{defn}
  For $i \in \cQ$, let $l_i^{f} \colon \Sp_{k,n} \to \Sp_{k,n}$ denote finite localization away from $L_{k,n}F(i+1)$. 
\end{defn}
\begin{rem}
  Because $L_{k,n}F(i+1)$ is in $\Sp^{\omega}_{k,n}$ by \Cref{thm:compact_objects}, this is a smashing localization.
\end{rem}
\begin{rem}
  
By \cite[Proposition 3.8]{Wolcott2015Variations} we have an equivalence of endofunctors of $\Sp_{k,n}$ (recall that $\langle \Tel(n) \rangle$ the Bousfield class of a telescope of a finite type $n$ spectrum)
\[
\begin{split}
l_i^f &\simeq L_{L_{k,n}\Tel(0) \vee L_{k,n} \Tel(1) \vee \cdots \vee L_{k,n}\Tel(i)}
\end{split}
\]
We note that $L_{k,n}\Tel(j)$ is trivial when $j \not \in \cQ$ by \cite[Proposition A.2.13]{Ravenel1992Nilpotence}. In particular, we have 
\[
\begin{split}
l_i^f &\simeq L_{L_{k,n}\Tel(k) \vee \cdots \vee L_{k,n}\Tel(i)}
\end{split}
\]
\end{rem}

We also consider the following Bousfield localization on $\Sp_{k,n}$. 
\begin{defn}
  For $i \in \cQ$, let $l_i \colon \Sp_{k,n} \to \Sp_{k,n}$ denote Bousfield localization at $K(k) \vee K(k+1) \vee \cdots \vee K(i)$. 
\end{defn}
\begin{rem}
  Following Wolcott \cite{Wolcott2015Variations} we consider the following variants of the telescope conjecture on $\Sp_{k,n}$ for $i \in \cQ$. 
$$\begin{array}{ll}
\LTCi: & \lr{L_{k,n}\Tel(i)} = \lr{K(i)} \mbox { in } \BL(\Sp_{k,n}). \\
\LTCii: & l_i^f X\stackrel{\sim}{\to} l_i X \mbox{ for all } X, \text{ or equivalently } \lr{\bigvee_{j=k}^i L_{k,n}\Tel(j)} = \lr{\bigvee_{j=k}^iK(j)}\mbox { in } \BL(\Sp_{k,n}).  \\
\LTCiii: & \mbox{If }X \mbox{  is a type }i\mbox{ spectrum and }f\mbox{  is a }v_i\mbox{  self-map, then } l_i (L_{k,n}X)\cong L_{k,n}(f^{-1}X).\\
\GSC: & \mbox{Every smashing localization is generated by a set of compact objects.}\\
\SDGSC: & \mbox{Every smashing localization is generated by a set of dualizable objects.}
\end{array}$$
Here $\LTC$ stands for the localized telescope conjecture, $\GSC$ is the generalized smashing conjecture, and $\SDGSC$ is the strongly dualizable generalized smashing conjecture. We emphasize the difference here because compact and dualizable objects do not coincide in $\Sp_{k,n}$ when $k \ne 0$. 
\end{rem}

\begin{prop}\label{prop:ltc}
  On $\Sp_{k,n}$, we have that $\LTCi,\LTCii$ and $\LTCiii$ hold for all $i \in \cQ$. 
\end{prop}
\begin{proof}
  By \cite[Theorem 3.12]{Wolcott2015Variations} it suffices to prove that $\LTCi$ holds, and by \Cref{prop:bousfield_lattice} this will follow if we show that $L_{k,n}\Tel(i)$ and $K(i)$ have the same support in $\Sp_{k,n}$. To see this, we have $\supp(K(i)) = \{ i \}$ by \Cref{ex:supp_kn}, while $\supp(\Tel(i)) = \{ i \}$ by \cite[Lemma 2.10 and Lemma 3.7]{Wolcott2015Variations}. 
\end{proof}
We now classify all smashing localizations on $\Sp_{k,n}$ and show that all variants of the telescope conjecture hold. 
\begin{thm}\label{thm:telescope}
  Let $L$ be a non-trivial smashing localization functor on $\Sp_{k,n}$, then $L \simeq l_j^f \simeq l_j$ for some $j \in \cQ$. In particular, the $\GSC$ and $\SDGSC$ both hold in $\Sp_{k,n}$. 
\end{thm}
\begin{proof}
 We closely follow \cite[Theorem 4.4]{Wolcott2015Variations}. Throughout the proof we let $\unit$ denote $L_{k,n}S^0$, the monoidal unit in $\Sp_{k,n}$, so that $\lr{L} = \lr{L\unit}$. By \eqref{eq:bousfield_decomp} we have 
  \[
\lr{L\unit} = \lr{\bigvee_{i \in \supp(L\unit)} K(i)}
  \]
Note that $\supp(L\unit)$ is non-empty because we assume $L \ne 0$. Hence, we can fix $j \in \supp(L\unit)$ so that $\lr{K(j)} \le \lr{L(\unit)}$ in $\BL(\Sp_{k,n})$. It follows that $L_{K(j)}L \simeq LL_{K(j)} \simeq L_{K(j)}$, and $\lr{L_{K(j)}\unit} = \lr{L_{K(j)}\unit \htimes L\unit} \le \lr{L\unit}$ in $\BL(\Sp_{k,n})$.  We also note that $L_{K(j)}\unit = L_{K(j)}L_{k,n}S^0 \simeq L_{K(j)}S^0$.

   By \cite[Proposition 5.3]{hovey_strickland99} we have $\langle L_{K(j)}S^0 \rangle = \bigvee_{i=0}^j \langle K(i) \rangle$ in $\BL(\Sp)$, and it follows easily that $\langle L_{K(j)}S^0 \rangle = \bigvee_{i=k}^j \langle K(i) \rangle$ in $\BL(\Sp_{k,n})$. It follows that $\lr{L\unit} \ge  \bigvee_{i=k}^j \langle K(i) \rangle$ in $\BL(\Sp_{k,n})$. We deduce that $\lr{L\unit} = \bigvee_{i=k}^j \lr{K(i)}$, where $j = \max\{\supp(L\unit)\}$, and hence by \Cref{prop:ltc} that $L \simeq l_j^f \simeq l_j$. Finally, because $L_nF(j+1)$ is compact and therefore also dualizable in $\Sp_{k,n}$, both the $\GSC$ and $\SDGSC$ hold in $\Sp_{k,n}$. 
\end{proof}
\begin{rem}
  Using \cite[Proposition 3.8.3]{HoveyPalmieriStrickl1997Axiomatic} and \Cref{thm:loc_coloc_classifying,thm:telescope} one can reprove the thick subcategory theorem \Cref{thm:thick_subcategory}. 
\end{rem}
\section{Descent theory and the \texorpdfstring{$E(n,J_k)$}{E(n,Jk)}-local Adams spectral sequence}\label{sec:descent_anss}
In this section we use descent theory to construct an Adams-type spectral sequence in the $E(n,J_k)$-local category. Using descent, we shall see that this has a vanishing line at some finite stage. Moreover, for $p \gg n$, we show that the $E(n,J_k)$-local Adams spectral sequence computing $\pi_*L_{k,n}S^0$ has a horizontal vanishing line on the $E_2$-page, and there are no non-trivial differentials. 

\subsection{Descendability}\label{sec:descendability}

We begin with the notion of a descendable object in an algebraic stable homotopy category.
\begin{rem}
  We recall that in $\cC$ there is an $\infty$-category $\CAlg(\cC)$ of commutative algebra objects, see \cite[Chapter 3]{lurie-higher-algebra}. Moreover, given $A \in \CAlg(\cC)$ we can define a stable, presentable, symmetric monoidal $\infty$-category $\Mod_A(\cC)$ of $A$-modules internal to $\cC$, with the relative $A$-linear tensor product \cite[Section 4.5]{lurie-higher-algebra}. We will mainly focus on the case $A = E_n$ and $\cC = \Sp_{k,n}$, so that $\Mod_{E_n}(\Sp_{k,n})$ denotes the $\infty$-category of $E(n,J_k)$-local $E_n$-modules, that is $E_n$-modules whose underlying spectrum is $E(n,J_k)$-local, with monoidal structure $A \wedge B = L_{k,n}(A \wedge_E B)$.  
\end{rem}
\begin{rem}
  Note that $E_n \in \CAlg(\Sp)$ by the Goerss--Hopkins--Miller \cite{goerss_hopkins} theorem, and so $E_n \in \CAlg(\Sp_{k,n})$ as well. On the other hand, $E(n,J_k)$ will not, in general, be a commutative ring spectrum (for example, $K(n)$ is never a commutative ring spectrum). 
\end{rem}
\begin{defn}(\cite[Definition 3.18]{Mathew2016Galois})
  A commutative algebra object $A \in \CAlg(\cC)$ is said to be descendable if $\unit \in \cC$ is $A$-nilpotent (\Cref{defn:nilpotence}), or equivalently $\cC = \thickt{A}$. 
\end{defn}
One reason to be interested in descendable objects is the following \cite[Proposition 3.22]{Mathew2016Galois}. 
\begin{prop}[Mathew]\label{prop:descendability_cat}
  Let $A \in \CAlg(\cC)$ be descendable, then the adjunction $C \leftrightarrows \Mod_{\cC}(A)$ given by tensoring with $A$ and forgetting is comonadic. In particular, the natural functor from $C$ to the totalization
  \[
\cC \to \Tot\left(\xymatrix@C=1em{\Mod_{A}(\cC) \ar@<0.5ex>[r] \ar@<-0.5ex>[r] & \Mod_{A \wedge A}(\cC)\ar@<1ex>[r] \ar@<-1ex>[r] \ar[r]&}\right)
  \]
  is an equivalence. 
\end{prop}
We also note the following \cite[Proposition 3.19]{Mathew2016Galois}.
\begin{prop}[Mathew]\label{prop:conservativity}
If $A \in \CAlg(\cC)$ is descendable, then the functor 
\[\cC \to \Mod_{A}(\cC), \quad M \mapsto M \wedge A\] is conservative.   
\end{prop}

\subsection{Morava modules and $L$-complete comodules}
The following theorem, essentially due to Hopkins--Ravenel \cite{Ravenel1992Nilpotence}, shows that the results of the previous section can be applied in $\Sp_{k,n}$. We note that $E_n \in \Sp$ is a commutative algebra object; this is the Goerss--Hopkins--Miller theorem \cite{goerss_hopkins}. It follows that $E_n \in \CAlg(\Sp_{k,n})$. 
\begin{thm}\label{thm:descendability_en}
  $E_n \in \CAlg(\Sp_{k,n})$ is descendable, and there is an equivalence of symmetric-monoidal stable $\infty$-categories
\begin{equation}\label{eq:tot_module}
\Sp_{k,n} \simeq \Tot\left(\xymatrix@C=1em{\Mod_{E_n}(\Sp_{k,n}) \ar@<0.5ex>[r] \ar@<-0.5ex>[r] & \Mod_{E_n \htimes E_n}(\Sp_{k,n})\ar@<1ex>[r] \ar@<-1ex>[r] \ar[r]&}\right)
\end{equation}
\end{thm}
\begin{proof}
  It is consequence of the Hopkins--Ravenel smash product theorem that $E_n \in \CAlg(\Sp_{n})$ is descendable, see \cite[Theorem 4.18]{Mathew2016Galois}. It follows from \cite[Corollary 3.21]{Mathew2016Galois} that $L_{k,n}E_n \simeq E_n$ is descendable in $\Sp_{k,n}$. The equivalence then follows from \Cref{prop:descendability_cat}. 
\end{proof}
By \Cref{prop:conservativity} we deduce the following. 
\begin{cor}\label{cor:conservative}
  The functor $E_n \htimes (-) \colon \Sp_{k,n} \to \Mod_{E_n}(\Sp_{k,n})$ is conservative. 
\end{cor}
We therefore define the following. 
\begin{defn}
  For $X \in \Sp_{k,n}$ the Morava module of $X$ is $(E_{k,n})^\vee_* X \coloneqq \pi_*(E_n \htimes X)$. 
\end{defn}

  We recall that $L_{k,n}X \simeq \varprojlim_j (L_nX \wedge M_k(j))$ (\Cref{prop:7.10}). The Milnor sequence then gives the following.  

\begin{lem}
  There is a short exact sequence
  \[
0 \to \varprojlim{}^1_{j} (E_n)_{\ast +1}(X \wedge M_j(k)) \to (E_{k,n})^{\vee}_*X \to \varprojlim{}_{j} E_*(X \wedge M_j(k)) \to 0
  \]
\end{lem}
\begin{ex}
  If $(E_n)_*X$ is a free $(E_n)_*$-module, then the $\varprojlim{}^1$ term vanishes and it follows that $(E_{k,n})^\vee_*X = (E_*X)^{\wedge}_{I_k}$.
\end{ex}
\begin{rem}
  As the short exact sequence shows, $(E_{k,n})^{\vee}_*X$ is not always complete with respect to the $I_k$-adic topology. However, it is always $L_0^{I_k}$-complete in the sense of \cite[Appendix A]{hovey_strickland99} - this the same argument as given in \cite[Proposition 8.4(a)]{hovey_strickland99}.
\end{rem}
\subsection{The $E(n,J_k)$-local $E_n$-Adams spectral sequence}
In this section we construct an Adams type spectral sequence in the $E(n,J_k)$-local category. When $k = 0$, this is the $E_n$-Adams spectral sequence, while when $k= n$ this is the $K(n)$-local $E_n$-Adams spectral sequence considered in \cite[Appendix A]{DevinatzHopkins2004Homotopy}.

To begin, we recall that the cobar (or Amitsur) complex for $E_n$ in $\Sp_{k,n}$ is 
\[
CB^\bullet(E_n) \colon \xymatrix@C=1em{ E_n \ar@<0.5ex>[r] \ar@<-0.5ex>[r]& E_n \htimes E_n \ar@<1ex>[r] \ar@<-1ex>[r]  \ar[r]& \cdots }
\]
\begin{defn}
   Let $\widehat \Ext^{s,\ast}_{(E_{k,n})^{\vee}_*(E_n)}((E_n)_*,(E_n)_*) \coloneqq H^s(\pi_\ast(CB^{\bullet}(E_n)))$, i.e., it is the cohomology of the complex
\[
\xymatrix@C=1em{ (E_n)_* \ar@<0.5ex>[r] \ar@<-0.5ex>[r]& (E_{k,n})^\vee_*(E_n) \ar@<1ex>[r] \ar@<-1ex>[r]  \ar[r]& \cdots }
\]
More generally, we let
\[
\widehat \Ext^{s,\ast}_{(E_{k,n})^{\vee}_*(E_n)}((E_n)_*,(E_{k,n})^\vee_*(X)) \coloneqq H^s(\pi_\ast(X \htimes CB^{\bullet}(E_n))).
\]
 \end{defn} 
\begin{prop}\label{prop:anss}
  For any spectrum $X$ there is a strongly convergent spectral sequence 
  \[
E_2^{s,t} \cong \widehat \Ext^{s,\ast}_{(E_{k,n})^{\vee}_*(E_n)}((E_n)_*,(E_{k,n})^\vee_*(X))  \implies \pi_*(L_{k,n}X)
  \]
  which has a horizontal vanishing line at a finite stage (independent of $X$). 
\end{prop}
\begin{proof}
This is the Bousfield--Kan spectral sequence associated to the tower $X \htimes CB^{\bullet}(E_n)$. The claimed results are a consequence of descendability (\Cref{thm:descendability_en}), see \cite[Corollary 4.4]{Mathew2016Galois} and \cite[Example 2.11, Proposition 2.12 and Proposition 2.14]{mathew_descent}. 
\end{proof}
\begin{rem}\label{rem:k(n)_local}
  For the $\Sp_{K(n)}$-local homotopy category, this completed $\Ext$ can be interpreted as an $\Ext$ group in the category of $L_0^{I_n}$-complete comodules \cite{BarthelHeard2016term}. In the case of $X = S^0$, Morava's change of rings theorem, in the form \cite[Theorem 4.3]{BarthelHeard2016term} shows that 
  \[
E_2^{s,t} \cong H^{s}_{c}(\bG_n,(E_n)_t),
  \]
  the continuous cohomology of the Morava stabilizer group $\bG_n$, and this spectral sequence is isomorphic to that considered by Devinatz and Hopkins in \cite[Appendix A]{DevinatzHopkins2004Homotopy}. The key point is the computation that 
  \[
(E_n)_*^{\vee}(E_n) \cong \Hom^{c}(\bG_n,(E_n)_*),
  \]
  for which see \cite{Hovey2004Operations}.  We remark that we do not know what $(E_{k,n})_*^{\vee}(E_n)$ is for $k \ne n$. However, the same arguments as in \cite{BarthelHeard2016term} go through; the pair $((E_n)_*,(E_{k,n})^{\vee}_*(E_n))$ is an $L_0^{I_k}$-complete comodule, and if $(E_{k,n})^{\vee}_*(X)$ is either a finitely-generated $(E_n)_*$-module, is the $I_k$-adic completion of a free-module, or has bounded $I_k$-torsion, then $(E_{k,n})^{\vee}_*(X)$ is a comodule over this Hopf algebroid, see \cite[Lemma 1.17 and Proposition 1.22]{BarthelHeard2016term}. The relative homological algebra studied in \cite[Section 2]{BarthelHeard2016term} also goes through to see that $\widehat \Ext^{s,\ast}$ as used above is a relative $\Ext$ group in the category of $L_0^{I_k}$-complete comodules. We will not use this in what follows, so we leave the details to the interested reader. 
\end{rem}
\begin{rem}
  In \cite[Section 7]{hms_pic} the authors construct the $K(n)$-local Adams spectral sequence for dualizable $K(n)$-local $X$ as the inverse limit of the $E_n$-Adams spectral sequences for $X \wedge M_n(j)$. The following result recovers the identification of the $E_2$-term in the case $k = n$. 
\end{rem}
\begin{prop}\label{prop:limit_ext}
  Let $M_k(j)$ be a tower of generalized Moore spectra of height $k$, then there is an isomorphism
  \[
  \widehat\Ext^{s,t}_{(E_{k,n})^{\vee}_*(E_n)}((E_n)_*,(E_n)_*) \cong \varprojlim_j \Ext^{s,t}_{(E_n)_*E_n}((E_n)_*,(E_n)_*(M_k(j))).
  \]
\end{prop}
\begin{proof}
 By definition $  \widehat\Ext^{s,t}_{(E_{k,n})^{\vee}_*(E_n)}((E_n)_*,(E_n)_*)$ is the cohomology of the complex
\[
\xymatrix@C=1em{ (E_n)_* \ar@<0.5ex>[r] \ar@<-0.5ex>[r]& (E_{k,n})^\vee_*(E_n) \ar@<1ex>[r] \ar@<-1ex>[r]  \ar[r]& \cdots }
\]
The $t$-th term of this complex is the homotopy of $L_{k,n}(E_n^{\wedge t}) \simeq \varprojlim_j(E_n^{\wedge t} \wedge M_k(j))$  by \Cref{prop:7.10}, and there is a corresponding Milnor exact sequence of the form 
\[
0 \to \varprojlim_j{}\!^1 \pi_{q+1}(E_n^{\wedge t} \wedge M_k(j)) \to \pi_q(L_{k,n}(E_n^{\wedge t})) \to \varprojlim_j \pi_q(E_n^{\wedge t} \wedge M_k(j)) \to 0.
\]
 We note that $E_n^{\wedge t}$ is Landweber exact, as the smash product of Landweber exact spectra (see \cite[Lemma 4.3]{BarthelStapleton2016Centralizers}). It follows that $\pi_*(E_n^{\wedge t} \wedge M_k(j)) \cong \pi_*(E_n^{\wedge t})/(p^{i_0},\ldots,u_{k-1}^{i_{k-1}})$ for suitable integers $i_0,\ldots,i_{k-1}$. In particular, the maps in the tower are surjections by the construction of the tower $\{ M_k(j) \}$ (see \Cref{rem:moore_spectra}), and so the $\varprojlim_j^1$-term vanishes, and 
 \[
(E_{k,n})^{\vee}_*(E_n^{t-1}) \cong  \pi_*(L_{k,n}(E_n^{\wedge t})) \cong \varprojlim_j \pi_q(E_n^{\wedge t} \wedge M_k(j)).
 \]
 Note that the cohomology of the complex $\{\pi_q(E_n^t \wedge M_k(j))\}_t$ is $\Ext_{(E_n)_*E_n}^{\ast,\ast}((E_n)_*,(E_n)_*(M_k(j)))$. Therefore, there is an exact sequence
 \[
 \begin{split}
0 \to \varprojlim_j{}\!^1 \Ext_{(E_n)_*E_n}^{q-1,\ast}((E_n)_*,(E_n)_*(M_k(j))) \to \widehat\Ext^{q,\ast}_{(E_{k,n})^{\vee}_*(E_n)}((E_n)_*,(E_n)_*)\\
\to \varprojlim_j \Ext_{(E_n)_*E_n}^{q,\ast}((E_n)_*,(E_n)_*(M_k(j))) \to 0. 
\end{split}
 \]
 We will see below in \Cref{cor:finiteness} that $\Ext_{(E_n)_*E_n}^{q,\ast}((E_n)_*,(E_n)_*(M_k(j)))$ is finite, and so the $\varprojlim_j^1$-term vanishes in the exact sequence, and the result follows. 
\end{proof}
\begin{rem}
  It follows that when $k \ne 0$, the groups $\widehat\Ext^{s,t}_{(E_{k,n})^{\vee}_*(E_n)}((E_n)_*,(E_n)_*)$ are profinite, i.e., either finite or uncountable. Contrast the case $k = 0$, where $\Ext_{(E_n)_*E_n}^{s,t}((E_n)_*,(E_n)_*)$ is countable \cite[Proof of Lemma 5.4]{Hovey1995Bousfield}.
\end{rem}
\subsection{Vanishing lines in the $E(n,J_k)$-local $E_n$-Adams spectral sequence}
In \Cref{prop:anss} we constructed a spectral sequence 
\[
\widehat \Ext^{s,t}_{(E_{k,n})^\vee_*(E_n)}((E_n)_*,(E_n)_*) \implies \pi_{t-s}(L_{k,n}S^0),
\]
and showed that, as a consequence of descendability, this has a horizontal vanishing line at some finite stage. In the extreme cases of $k = 0$ and $k = n$ it is known that when $p \gg n$ and $X = S^0$, this vanishing line occurs on the $E_2$-page, and occurs at $s = n^2+n$ and $s = n^2$, respectively, see \cite[Theorem 5.1]{HoveySadofsky1999Invertible} and \cite[Theorem 6.2.10]{ravenel_green}. In this section, we show (\Cref{thm:vanishing}) that the analogous result occurs in general; for $p \gg n$ there is a vanishing line on the $E_2$-page of the spectral sequence of \Cref{prop:anss} above $s = n^2+n-k$ in the case $X = S^0$. The proof relies on a variant of the chromatic spectral sequence \cite[Chapter 5]{ravenel_green}, which we now construct. Along the way we prove \Cref{cor:finiteness}, which also completes the proof of \Cref{prop:limit_ext}. 
\begin{rem}[The chromatic spectral sequence]
  Fix $k \le n$, and for $0 \le s \le n-k$ let $M^s$ denote the $(E_n)_*(E_n)$-comodule
  \[
u_{k+s}^{-1}(E_n)_*/(p,u_1,\ldots,u_{k-1},u_{k}^{\infty},\ldots,u_{k+s-1}^{\infty})
  \]
Arguing as in \cite[Lemma 5.1.6]{ravenel_green} there is an exact sequence of $(E_n)_*(E_n)$-comodules
  \[
(E_n)_*/I_k \to M^0 \to M^1 \to \ldots \to M^{n-k} \to 0.
  \] 
  Applying \cite[Theorem A.1.3.2]{ravenel_green} there is then a chromatic spectral sequence of the form 
  \begin{equation}\label{eq:css}
E_1^{s,r,\ast} \cong  \Ext^{r,\ast}_{(E_n)_*(E_n)}((E_n)_*,M^s)\implies \Ext^{r+s,\ast}_{(E_n)_*(E_n)}((E_n)_*,(E_n)_*/I_k)
  \end{equation}
\end{rem}
\begin{prop}\label{thm:css}
  In the chromatic spectral sequence \eqref{eq:css} we have
  \[
E_1^{s,r,\ast} \cong \begin{cases}
  \Ext^{r,\ast}_{(E_{k+s})_*(E_{k+s})}((E_{k+s})_*,(E_{k+s})_*/(p,\ldots,u_{k-1},u_{k}^{\infty},\ldots,u_{k+s-1}^{\infty})) & s \le n-k \\
  0 & s > n-k. 
\end{cases}
  \]
  If particular, if $p-1$ does not divide $k+s$, we have $E_1^{s,r,\ast} = 0$ for $r > (s+k)^2$. Thus, if $p-1$ does not divide $k+s$ for all $0 \le s \le n-k$,\footnote{Taking $p>n+1$ suffices, but may not be optimal.}  then
  \[
\Ext^{s,\ast}_{(E_n)_*(E_n)}((E_n)_*,(E_n)_*/I_k) = 0
  \]
  for $s > n^2 + n-k$.
\end{prop}
\begin{proof}
  This is similar to the proof by Hovey and Sadofsky \cite[Theorem 5.1]{HoveySadofsky1999Invertible}, which is the case where $k = 0$. We first recall the change of rings theorem of Hovey and Sadofsky \cite[Theorem 3.1]{HoveySadofsky1999Invertible}; if $M$ is a $BP_*BP$-comodule, on which $v_j$ acts isomorphically, and $n \ge j$, then there is an isomorphism\footnote{Hovey and Sadofsky work with $E(n)$ instead of $E_n$, but this does not change anything in light of \cite[Theorem C]{HoveyStrickl2005Comodules}.}
  \[
\Ext_{BP_*BP}^{*,*}(BP_*,M) \cong \Ext^{*,*}_{(E_n)_*(E_n)}((E_n)_*,(E_n)_*\otimes_{BP_*}M).
  \]

  Applying this change of rings theorem twice to the $BP_*BP$-comodule 
  \[
  u_{k+s}^{-1}BP_*/(p,\ldots,u_{k-1},u_k^{\infty},\ldots,u_{k+s-1}^{\infty})\]
   with $j = k+s$ and $j = n$ shows that the $E_1$-term has the claimed form.

  For brevity, let us denote $I = (p,\ldots,u_{k-1},u_{k}^{\infty},\ldots,u_{k+s-1}^{\infty})$. By Morava's change of rings theorem, we have 
  \[
\Ext^{r,\ast}_{(E_{k+s})_*(E_{k+s})}((E_{k+s})_*,(E_{k+s})_*/I) \cong H^r(\mathbb{G}_{k+s},(E_{k+s})_*/I).
  \]
 Morava's vanishing theorem (\cite[Theorem 6.2.10]{ravenel_green}) shows that if $p-1$ does not divide $k+s$, then 
 \[
H^r(\mathbb{G}_{k+s},(E_{k+s})_*) = 0
 \]
 for $r>(k+s)^2$. Along with an argument similar to that given by Hovey and Sadofsky's, using standard exact sequences and taking direct limits we find that 
  \[
\Ext^{r,\ast}_{(E_{k+s})_*(E_{k+s})}((E_{k+s})_*,(E_{k+s})_*/I) = 0
  \]
  for $r > (k+s)^2$ as well. 
\end{proof}
\begin{rem}\label{rem:css_alternative}
 Let $M_k$ denote a generalized Moore spectrum of type $k$, then there is an obvious analog of this spectral sequence, whose abutment is
  \[\Ext^{r+s,\ast}_{(E_n)_*(E_n)}((E_n)_*,(E_n)_*(M_k)) \cong \Ext^{r+s,\ast}_{(E_n)_*(E_n)}((E_n)_*,(E_n)_*/(p^{i_0},\ldots,u_{k-1}^{i_{k-1}}))\]
   with $E_1$-term of the form 
     \[
E_1^{s,r,\ast} \cong \begin{cases}
  \Ext^{r,\ast}_{(E_{k+s})_*(E_{k+s})}((E_{k+s})_*,(E_{k+s})_*/(p^{i_0},\ldots,u_{k-1}^{i_{k-1}},u_{k}^{\infty},\ldots,u_{k+s-1}^{\infty})) & s \le n-k \\
  0 & s > n-k. 
\end{cases}
  \]
\end{rem}
\begin{rem}
  The following completes the proof of \Cref{prop:limit_ext}.
\end{rem}
\begin{cor}\label{cor:finiteness}
  \sloppy Let $M_k$ denote a generalized Moore spectrum of type $k$, then the group $\Ext^{r,\ast}_{(E_n)_*(E_n)}((E_n)_*,(E_n)_*(M_k))$ is finite. 
\end{cor}
\begin{proof}
  By taking appropriate exact sequences it suffices to show this for $(E_n)_*/I_k$ (alternatively, one can argue directly using the spectral sequence of \Cref{rem:css_alternative}). Given the chromatic spectral sequence, we can reduce to showing that $H^r(\mathbb{G}_{k+s},(E_{k+s})_*/I)$ is finite, with $I$ as in the proof of the previous proposition. For this, see \cite[Proposition 4.2.2]{symonds_weigel}.
\end{proof}
\begin{cor}\label{cor:vanishing_line}
  Let $M_k$ denote a generalized Moore spectrum of type $k$, then if $p-1$ does not divide $k+s$ for $0 \le s \le n-k$ we have 
  \[
\Ext^{s,\ast}_{(E_n)_*(E_n)}((E_n)_*,(E_n)_*(M_k)) = 0
  \]
for $s > n^2+n-k$. 
\end{cor}
\begin{proof}
  Recall that $(E_n)_*(M_k) \cong (E_n)_*/(p^{i_0},\ldots,u_{k-1}^{i_{k-1}})$ for a suitable sequence $(i_0,\ldots,i_{k-1})$ of integers. The result for the sequence $(1,\ldots,1)$ holds by \Cref{thm:css} and therefore in general by taking appropriate exact sequences.  
\end{proof}
\begin{thm}\label{thm:vanishing}
  Suppose $p-1$ does not divide $k+s$ for $0 \le s \le n-k$, then 
  \[
 \widehat \Ext^{s,t}_{(E_{k,n})^{\vee}_*(E_n)}((E_n)_*,(E_n)_*) = 0
  \]
  for $s > n^2+n-k$. 
\end{thm}
\begin{proof}
  Combine \Cref{prop:limit_ext,cor:vanishing_line}.
\end{proof}
\begin{rem}
  The condition on the prime is always satisfied if $p$ is large enough compared to $n$ (in fact $p > n+1$ suffices). This suggests the following, which we do not attempt to make precise: for large enough primes, the cohomological dimension of $(E_n)_*$ in a suitable category of (completed)-$(E_{k,n})_*^\vee (E_{n})$-comodules is finite, and equal to $n^2+n-k$. 
\end{rem}
We also have the following expected sparseness result.
\begin{prop}\label{thm:sparseness}
Let $q = 2(p-1)$, then
  \[
  \widehat \Ext^{s,t}_{(E_n)^{\vee}_*(E_n)}((E_n)_*,(E_n)_*) = 0
  \]
  for all $s$ and $t$ unless $t \equiv 0 \mod{q}$. Consequently, in the spectral sequence of \Cref{prop:anss}, $d_r$ is nontrivial only if $r \equiv 1 \mod{q}$ and $E_{mq+2}^{\ast,\ast} = E_{mq+q+1}^{\ast,\ast}$ for all $m\ge 0$.  
\end{prop}
\begin{proof}
  Using \Cref{prop:limit_ext} it suffices to show the first statement for the $E_1$-term of the chromatic spectral sequence of \Cref{thm:css}. Again using the Hovey--Sadofsky change of rings theorem this $E_1$-term is isomorphic to
  \[
\Ext^{s,\ast,\ast}_{BP_*BP}(BP_*,u_{k+s}^{-1}BP_*/(p,\ldots,u_{k-1},u_k^{\infty},\ldots,u_{k+s-1}^{\infty})).
  \]
  Now apply \cite[Proposition 4.4.2]{ravenel_green}. 
\end{proof}

\section{Dualizable objects in \texorpdfstring{$\Sp_{k,n}$}{Spk,n}}\label{sec:dualizable}
In this section we use descendability to characterize the dualizable objects in $\Sp_{k,n}$. As noted previously, as long as $k \ne 0$, these differ from the compact objects studied in \Cref{sec:compact_objects}. 
\begin{defn}
  Let $(\cC,\wedge,\unit)$ be a symmetric-monoidal $\infty$-category, then $X \in \cC$ is dualizable if there exists an object $D_{\cC}X$ and a pair of morphisms
  \[
e \colon D_{\cC}X \wedge X \to \unit \quad \text{ and } c \colon \unit \to X \wedge D_{\cC}X
  \]
  such that the composites
  \[
X \xrightarrow{c \wedge \text{id}} X \wedge D_{\cC}X \wedge X \xrightarrow{\text{id} \wedge e} X
  \]
  and
  \[
D_{\cC}X \xrightarrow{\text{id} \wedge c} D_{\cC}X \wedge X \wedge D_{\cC}X \xrightarrow{e \wedge \text{id}} D_{\cC}X
  \]
  are the identity on $X$ and $D_{\cC}X$, respectively. 
\end{defn}
\begin{rem}\label{rem:dual_homotopy_category}
  The definition makes it clear that $X \in \cC$ is dualizable if and only if it is dualizable in the homotopy category of $\cC$. Moreover, a formal argument shows that, if it exists, we must have $D_{\cC}X \simeq F(X,\unit)$. Finally, for the equivalence with other definitions of dualizability the reader may have seen, see \cite[Theorem 1.3]{dold_puppe}. 
\end{rem}
\begin{defn}
  We let $\cC^{\dual} \subseteq \cC$ denote the full subcategory consisting of the dualizable objects of $\cC$.
\end{defn}
\begin{rem}
  The full subcategory $\cC^{\dual}$ is a thick tensor ideal \cite[Theorem A.2.5]{HoveyPalmieriStrickl1997Axiomatic}. 
\end{rem}
We have the following relationship between descent theory and dualizability. 
\begin{prop}\label{prop:dualiable_descent}
  Let $A \in \CAlg(\cC)$ be descendable, then the adjunction $C \leftrightarrows \Mod_{A}(\cC)$ gives rise to an equivalence of symmetric monoidal $\infty$-categories
  \[
\cC^{\dual} \to \Tot\left(\xymatrix@C=1em{\Mod_{A}(\cC)^{\dual} \ar@<0.5ex>[r] \ar@<-0.5ex>[r] & \Mod_{A \wedge A}(\cC)^{\dual}\ar@<1ex>[r] \ar@<-1ex>[r] \ar[r]&}\right).
  \]
  In particular, $M \in \cC$ is dualizable if and only if $M \wedge A \in \Mod_A(\cC)$ is dualizable.  
\end{prop}
\begin{proof}
  The first claim follows from \Cref{prop:descendability_cat} because passing to dualizable objects commutes with limits of $\infty$-categories
  \cite[Proposition 4.6.1.11]{lurie-higher-algebra}. The second is then an easy consequence, using that all the maps in the totalization are symmetric monoidal.  
\end{proof}
\subsection{Dualizable objects in the $E(n,J_k)$-local category}
Using \Cref{thm:descendability_en,prop:dualiable_descent} we deduce the following. 
\begin{prop}\label{prop:dualiable_descent_en}
  The adjunction $\Sp_{k,n} \leftrightarrows \Mod_{E_n}(\Sp_{k,n})$ gives rise to an equivalence of symmetric monoidal $\infty$-categories
  \[
\Sp_{k,n}^{\dual} \to \Tot\left(\xymatrix@C=1em{\Mod_{E_n}(\Sp_{k,n})^{\dual} \ar@<0.5ex>[r] \ar@<-0.5ex>[r] & \Mod_{E_n \htimes E_n}(\Sp_{k,n})^{\dual}\ar@<1ex>[r] \ar@<-1ex>[r] \ar[r]&}\right).
  \]
   In particular, $X \in \Sp_{k,n}$ is dualizable if and only if $E_n \htimes X \in \Mod_{E_n}(\Sp_{k,n})$ is dualizable.  
\end{prop}
This proposition suggests we begin by studying dualizable objects in the category $\Mod_{E_n}(\Sp_{k,n})$. Fortunately, these have a nice characterization. We begin with the following. 
\begin{lem}\label{lem:K(n)local}
  If $X$ is dualizable in $\Mod_{E_n}(\Sp_{k,n})$ then the spectrum underlying $X$ is $K(n)$-local. 
\end{lem}
\begin{proof}
We first note that for any $M \in \Mod_{E_n}$ (in particular, for $M = X$) the Bousfield localization $L_{K(n)}M$ is the spectrum underlying $L_{E_n \wedge X}^{E_n}$, where the latter denotes the Bousfield localization with respect to $E_n \wedge X$ internal to the category of $E_n$-modules. In particular, the localization map $M \to L_{K(n)}M$ is a map in $\Mod_{E_n}$. See Chapter VIII of \cite{ekmm}, in particular \cite[Proposition VIII.1.8]{ekmm}. If follows that $K(n)$-localization defines a localization
\[
L_{K(n)} \colon \Mod_{E_n}(\Sp_{k,n}) \to \Mod_{E_n}(\Sp_{K(n)}).
\]
Because $X \in \Mod_{E_n}(\Sp_{k,n})^{\dual}$, using \cite[Lemma 3.3.1]{HoveyPalmieriStrickl1997Axiomatic}, we see that there are equivalences
\[
L_{K(n)}X \simeq L_{k,n}(( L_{K(n)}E_n) \wedge_{E_n} X) \simeq L_{k,n}(E_n \wedge_{E_n} X) \simeq L_{k,n}X \simeq X,
\]
as claimed. 
\end{proof}
\begin{rem}
  For the following, we let $K_n \cong E_n/I_n$. This is a 2-periodic form of Morava $K$-theory; indeed, 
  \[
(K_n)_*X \cong (K_n)_* \otimes_{K(n)_*}K(n)_*X,
  \]
  and so $\langle K(n) \rangle = \langle K_n \rangle$. We use this only because $K_n$ is naturally an $E_n$-module. 
\end{rem}
\begin{prop}\label{prop:dualizable_local}  
  For $X \in \Mod_{E_n}(\Sp_{k,n})$ the following are equivalent:
  \begin{enumerate}
    \item $X$ is dualizable in $\Mod_{E_n}(\Sp_{k,n})$.
    \item $X$ is compact (equivalently, dualizable) in $\Mod_{E_n}(\Sp)$. 
    \item The spectrum underlying $X$ is $K(n)$-local, and the homotopy groups $\pi_*(K_n \wedge_{E_n} X)$ are finite. 
  \end{enumerate}
\end{prop}
\begin{proof}
We first show that (2) implies (1). The compact objects in $\Mod_{E_n}(\Sp)$ are precisely those in the thick subcategory generated by $E_n$ (for example, \cite[Proposition 7.2.4.2]{lurie-higher-algebra}). Since $E_n \in \Mod_{E_n}(\Sp_{k,n})^{\dual}$, and the collection of dualizable objects is thick, the implication (2) implies (1) follows.  

  Conversely, assume that (1) holds. As above we have a symmetric-monoidal localization
  \[
L_{K(n)} \colon \Mod_{E_n}(\Sp_{k,n}) \to \Mod_{E_n}(\Sp_{K(n)}),
\]
which preserves dualizable objects (as any symmetric-monoidal functor does). Using \Cref{lem:K(n)local} it follows that $L_{K(n)}X \simeq X$ is dualizable in $\Mod_{E_n}(\Sp_{K(n)})$, which implies by \cite[Proposition 10.11]{Mathew2016Galois} that $X$ is compact in $\Mod_{E_n}(\Sp)$. 

  Finally, the equivalence of (2) and (3) is well-known, see for example \cite[Proposition 2.9.4]{lurie-brauer}.\footnote{We note that Lurie has confirmed via private communication that the cited proposition \cite[Proposition 2.9.4]{lurie-brauer} should additionally have the assumption $X$ is $K(n)$-local in Condition (3).} 
\end{proof}
\begin{rem}
  Suppose $X \in \Sp_{k,n}^{\dual}$, so that $L_{k,n}(E_n \wedge X) \in \Mod_{E_n}(\Sp_{k,n})^{\dual}$. The previous proposition then implies that $L_{k,n}(E_n \wedge X) \simeq L_{K(n)}L_{k,n}(E_n \wedge X) \simeq L_{K(n)}(E_n \wedge X)$. In other words, for dualizable $X$, there is an isomorphism $(E_{k,n})^\vee_*(X) \cong (E_{n,n})^{\vee}_*(X)$. 
\end{rem}
We now give our characterization of dualizable spectra in $\Sp_{k,n}$ - see \cite[Theorem 8.6]{hovey_strickland99} for the case $k = n$. We note that even in this case our proof, which uses descendability, differs from that of Hovey and Strickland. 
\begin{thm}\label{thm:dualspkn}
  The following are equivalent for $X \in \Sp_{k,n}$:
  \begin{enumerate}
    \item $X$ is dualizable.
    \item $X$ is $F$-small, i.e., for any collection of objects $\{ Z_i \}$, the natural map $$L_{k,n}(\bigvee F(X,Z_i)) \to F(X,L_{k,n}(\bigvee Z_i))$$ is an equivalence. 
    \item $E_n \htimes X \in \Mod_{E_n}(\Sp_{k,n})$ is dualizable. 
    \item $E_n \htimes X \in \Mod_{E_n}(\Sp)$ is dualizable (equivalently, compact).
    \item $E_n \htimes X$ is $K(n)$-local and $(K_n)_*X$ is finite. 
    \item $(E_{k,n})^{\vee}_*(X)$ is a finitely generated $E_*$-module. 
  \end{enumerate}
\end{thm}
\begin{proof}
  The equivalences between the first five items come from \cite[Theorem 2.1.3(c)]{HoveyPalmieriStrickl1997Axiomatic} ($(1) \iff (2)$), \Cref{prop:dualiable_descent} ($(1) \iff (3)$) and \Cref{prop:dualizable_local} ($(3) \iff (4) \iff (5)$).  Finally, we note that if $M$ is an $E_n$-module, then $M$ is compact if and only if $\pi_*M$ is finitely generated over $(E_n)_*$, see \cite[Lemma 10.2(i)]{Greenlees2018Homotopy}. Applying this with $M=E_n \htimes X$ gives the equivalence between $(4)$ and $(6)$.
 \end{proof}
 Finally, we show that there is only a set of isomorphism classes of dualizable objects. 
 \begin{lem}\label{lem:essentially_small}
   There are at most $2^{\aleph_0}$ isomorphism classes of objects in $\Sp_{k,n}^{\dual}$. 
 \end{lem}
 \begin{proof}
   This is the same as the argument given in \cite[Propositon 12.17]{hovey_strickland99}. Namely, there are only countably many finite spectra $X'$ of type at least $k$, and for each one $[L_nX',L_nX']$ is finite, so $L_nX'$ has only finitely many retracts. By \Cref{thm:compact_objects} it follows that there is a countable set of isomorphism classes of objects in $\Sp_{k,n}^{\omega}$. If $U$ and $V$ are finite, then $[U,V]$ is finite, and so there are at most $\aleph_0^{\aleph_0} = 2^{\aleph_0}$ different towers of spectra in $\Sp_{k,n}^{\omega}$. For $X \in \Sp_{k,n}^{\dual}$, write $X \simeq \varprojlim_j X \wedge M_k(j)$, as in \Cref{prop:7.10}. Because $X$ is dualizable and $M_k(j)$ is compact, $X \wedge M_k(j)$ is compact \cite[Theorem 2.1.3(a)]{HoveyPalmieriStrickl1997Axiomatic}. Therefore, $X$ is the inverse limit of a tower of spectra in $\Sp_{k,n}^{\omega}$, and hence there are at most $2^{\aleph_0}$ isomorphism classes of objects in $\Sp_{k,n}^{\dual}$. 
 \end{proof}
\subsection{The spectrum of dualizable objects}\label{sec:balmer_dual}
In \Cref{thm:thick_subcategory} we computed the thick subcategories (equivalently, thick tensor-ideals) of compact objects in $\Sp_{k,n}$. One could also ask for a classification of the thick tensor-ideals of dualizable objects in $\Sp_{k,n}$, or equivalently a computation of the Balmer spectrum $\Spc(\Sp_{k,n}^{\dual})$ (which is well-defined by \Cref{lem:essentially_small}).  Based on a conjecture of Hovey and Strickland, the author, along with Barthel and Naumann, investigated $\Spc(\Sp_{K(n)}^{\dual})$ in detail in \cite{bhn}, showing that the Hovey--Strickland conjecture holds when $n = 1$ and $2$, and that in general it is implied by a hope of Chai in arithmetic geometry. In this section, we make some general comments regarding $\Spc(\Sp_{k,n}^{\dual})$. 
\begin{rem}
  The following full subcategories were considered in the case $k = n$ by Hovey--Strickland \cite[Definition 12.14]{hovey_strickland99}. 
\end{rem}
\begin{defn}
  For $i \le n$, let $\cal D_{i}$ denote the category of spectra $X \in \Sp_{k,n}^{\dual}$ such that $X$ is a retract of $Y \wedge Z$ for some $Y \in \Sp_{k,n}^{\dual}$ and some finite spectrum $Z$ of type at least $i$. It is also useful to set $\cD_{n+1} = (0)$. 
\end{defn}
\begin{rem}
  We note that $\cD_k \simeq \Sp_{k,n}^{\omega}$; this is a consequence of the characterization of compact objects given in \Cref{thm:compact_objects}, and that $\cD_0 = \Sp_{k,n}^{\dual}$. 
\end{rem}
The following is \cite[Proposition 4.17]{hovey_strickland99}
\begin{lem}
  $X$ is in $\cD_k$ if and only if $X$ is a module over a generalized Moore spectrum of type $k$. Moreover, $\cD_k \subseteq \Sp_{k,n}^{\dual}$ is a thick tensor ideal. 
\end{lem}
Hovey and Strickland conjecture that in the case $k = n$ these exhaust the thick-tensor ideals of $\Sp_{K(n)}^{\dual}$. This has been investigated in detail in \cite{bhn}. We conjecture this holds more generally in $\Sp_{k,n}$. 
\begin{conj}\label{conj:gen_hovey_strickland}
 If $\cC$ is a thick tensor-ideal of $\Sp_{k,n}^{\dual}$, then $\cC = \cD_i$ for some $0 \le i \le n+1$. Equivalently, 
 \[
\Spc(\Sp_{k,n}^{\dual}) = \{ \cD_1,\ldots,\cD_{n+1}\}
 \]
 with topology determined by the closure operator $\overline{\{ \cD_i \}} = \{ \cD_j \mid j \ge i \}$. 
\end{conj}
In this section we show that if \Cref{conj:gen_hovey_strickland} holds $K(n)$-locally, i.e., for  $\Sp_{n,n}^{\dual}$, then it holds for all $\Sp^{\dual}_{k,n}$. We first recall the following definition.
\begin{defn}
  Suppose $F \colon \cal{K} \to \cal{L}$ is an exact tensor triangulated functor between tensor-triangulated categories. We say that $F$ detects tensor-nilpotence of morphisms if every morphism $f \colon X \to Y$ in $\cal{K}$ such that $F(f) = 0$ satisfies $f^{\otimes m} = 0$ for some $m \ge 1$. 
\end{defn}
We will use the following.
\begin{prop}\label{prop:tensor_nilpotence}
  Suppose $A \in \CAlg(\cal C)$ is descendable, then extension of scalars $\cal C \to \Mod_A(\cal C)$ detects tensor-nilpotence of morphisms. 
\end{prop}
\begin{proof}
  Let $I$ denote the fiber of $\unit \xrightarrow{\eta} A$, and let $\xi \colon I \to \unit$ denote the induced map. By \cite[Proposition 4.7]{mnn_nilpotence} if $A$ is descendable, then there exists $m \ge 1$ such that $I^{\otimes m} \to \unit$ is null-homotopic, i.e., $\xi$ is tensor-nilpotent. We can now argue as in (ii) implies (iii) of \cite{Balmer2016Separable}: suppose we are given $f \colon X \to Y$ a morphism in $\cal C$ with $A \otimes f \colon A \otimes X \to A \otimes Y$ null-homotopic. Now consider the diagram of fiber sequences:
\[\begin{tikzcd}
  {I \otimes X} & X & {A \otimes X} \\
  {I \otimes Y} & Y & {A \otimes Y}
  \arrow["{\xi \otimes \text{id}_X}", from=1-1, to=1-2]
  \arrow["{\eta \otimes \text{id}_X}", from=1-2, to=1-3]
  \arrow["{\xi \otimes \text{id}_Y}"', from=2-1, to=2-2]
  \arrow["{\eta \otimes \text{id}_Y}"', from=2-2, to=2-3]
  \arrow["{\text{id}_I \otimes f}"', from=1-1, to=2-1]
  \arrow["f", from=1-2, to=2-2]
  \arrow["{\text{id}_A \otimes f}", from=1-3, to=2-3]
\end{tikzcd}\]
We see that $(\eta \otimes \text{id}_Y)f$ is null-homotopic, so $f$ factors through $\xi \otimes \text{id}_Y$, which is tensor-nilpotent. 
\end{proof}
The following is our key observation. 
\begin{prop}\label{prop:surjective_balmer}
  Let $i > k$, then the map induced by localization
  \[
\Spc(\Sp_{i,n}^{\dual}) \to \Spc(\Sp_{k,n}^{\dual})
  \]
  is surjective. 
\end{prop}
\begin{proof}
  By \cite[Theorem 1.3]{balmer_surjective} it suffices to show that the functor $L_{i,n} \colon \Sp_{k,n}^{\dual} \to \Sp_{i,n}^{\dual}$ detects tensor-nilpotence of morphisms. To that end, let $f \colon X \to Y$ be a morphism in $\Sp_{k,n}^{\dual}$ with $L_{i,n}(f) = 0$, so that we must show $f^{\htimes m} = 0$ for some $m \ge 1$. Because $E_n \in \Sp_{k,n}$ is descendable, \Cref{prop:tensor_nilpotence} shows that 
  \[L_{k,n}(E_n \wedge -) \colon \Sp_{k,n} \to \Mod_{E_n}(\Sp_{k,n})
  \] detects tensor-nilpotence of morphisms, and hence so does its restriction to dualizable objects, i.e., $L_{k,n}(E_n \wedge f) = 0 \implies f^{\htimes m} = 0$ for some $m \ge 1$. In other words, it suffices to show that $L_{k,n}(E_n \wedge f)$ is trivial. By \Cref{lem:K(n)local} however, this is a morphism in $\Mod_{E_n}(\Sp_{n,n})$. In particular, $L_{k,n}(E_n \wedge f) \simeq L_{i,n}(E_n \wedge f) \simeq L_{i,n}(E_n \wedge L_{i,n}(f)) = 0$ because $L_{i,n}(f) = 0$ by assumption. 
\end{proof}
\begin{thm}\label{thm:hovey_strickland_conjecture}
  Suppose \Cref{conj:gen_hovey_strickland} holds for $\Sp_{n,n}^{\dual}$, then it holds for all $\Sp_{k,n}^{\dual}$. 
\end{thm}
\begin{proof}
  By \cite[Proposition 3.5]{bhn} \Cref{conj:gen_hovey_strickland} holds for $\Sp_{n,n}^{\dual}$ if and only if $L_{n,n} \colon \Sp_{0,n}^{\dual} \to \Sp_{n,n}^{\dual}$ induces a homeomorphism on Balmer spectra. In other words, the composite, induced by the localization maps,
  \[
\Spc(\Sp_{n,n}^{\dual}) \to \Spc(\Sp_{n-1,n}^{\dual}) \to \cdots \to \Spc(\Sp_{0,n}^{\dual})
  \] 
  is a homeomorphism. It follows that $\Spc(L_{n-1,n}) \colon \Spc(\Sp_{n,n}^{\dual}) \to \Spc(\Sp_{n-1,n}^{\dual})$ is an injection, and hence a bijection by \Cref{prop:surjective_balmer}. Using that $\Spc(L_{n-1,n})$ is continuous and the topologies on each space, we see that it is fact a homeomorphism. It follows that $\Spc(\Sp_{n-1,n}^{\dual}) \to \Spc(\Sp_{0,n}^{\dual})$ is a homeomorphism, and we can now repeat the argument. 
\end{proof}
By \cite[Theorem 4.15]{bhn} \Cref{conj:gen_hovey_strickland} holds for $\Sp_{2,2}^{\dual}$. Along with \Cref{thm:hovey_strickland_conjecture} we deduce the following. 
\begin{cor}\label{cor:height_2_hovey_strickland}
  The Balmer spectrum $\Spc(\Sp_{1,2}^{\dual}) = \{ \cal{D}_1,\cal{D}_2,\cal{D}_3 \} $ with $(0) = \cal{D}_3 \subsetneq \cal{D}_2 \subsetneq \cal{D}_1 = \Sp_{1,2}^{\omega}$. In particular, if $\cC$ is a thick tensor-ideal of $\Sp_{1,2}^{\dual}$, then $\cC = \cal{D}_k$ for $0 \le k \le 3$. 
\end{cor}
\section{The Picard group of the \texorpdfstring{$\Sp_{k,n}$}{Spk,n}-local category}\label{sec:picard}
In this section we study invertible objects in the $\Sp_{k,n}$-local category. We show that invertibility of an object can be detected by its Morava module. We construct a spectral sequence computing the homotopy groups of the Picard spectrum of $\Sp_{k,n}$ and use this to show that if $p$ is large compared to $n$, then the Picard group of $\Sp_{k,n}$ is entirely algebraic, in a sense we make precise.  
\subsection{Invertible objects and Picard spectra}
We recall that if $\cC$ is a symmetric monoidal category we denote by $\Pic(\cC)$ the group of isomorphism classes of invertible objects; a priori this could be a proper class, but if $\cC$ is a presentable stable $\infty$-category (which it will always be in our cases), then it is actually a set \cite[Remark 2.1.4]{mathew_pic}. 

The following standard lemma will be useful for us. Here we write $D_{\cC}(X)$ for the dual of an object in a category $\cC$, i.e., $D_{\cC}(X) = F(X,\unit)$. Note that an invertible object is always dualizable \cite[Proposition A.2.8]{HoveyPalmieriStrickl1997Axiomatic}. 
\begin{lem}\label{lem:invertible_conservative}
  Let $F \colon \cC \to \cD$ be a symmetric-monoidal conservative functor between stable $\infty$-categories, then $X \in \cC$ is invertible if and only if $F(X) \in \cD$ is invertible. 
\end{lem}
\begin{proof}
  We first note that $X$ is invertible if and only if the natural morphism $X \otimes_{\cC} D_{\cC}(X) \to \mathbf{1}_{\cC}$ is an equivalence, see \cite[Proposition A.2.8]{HoveyPalmieriStrickl1997Axiomatic}. Because $F$ is assumed to be symmetric monoidal and conservative, this is an equivalence if and only if it is so after applying $F$, i.e., if and only if $F(X) \otimes_{\cD} F(D_{\cC}(X)) \to \mathbf{1}_{\cD}$ is an equivalence. But $F(D_{\cC}(X)) \simeq D_{\cD}(F(X))$, as symmetric-monoidal functors preserve dualizable objects, and the result follows. 
  \end{proof}
\begin{rem}\label{rem:pic_spectrum_limits}
  To our symmetric monoidal category $\cC$ we can instead associate the Picard spectrum $\pics(\cC)$ \cite[Definition 2.2.1]{mathew_pic}; this is a connective spectrum with the property that 
  \[
\pi_i(\pics(\cC)) = \begin{cases}
\Pic(\cC) & i = 0 \\
(\pi_0(\End_{\cC}(\unit))^{\times} & i = 1 \\
\pi_{i-1}(\End_{\cC}(\unit)) & i > 1.
\end{cases}
  \]
  The key advantage of using the Picard spectrum is that, as a functor from the $\infty$-category  of symmetric monoidal $\infty$-categories to the $\infty$-category of connective spectra, $\pics$ commutes with limits \cite[Proposition 2.2.3]{mathew_pic}.
\end{rem}
\begin{ex}\label{ex:localized_pic}
  Let $\cC$ be a category, $A \in \CAlg(\cC)$. Then the Picard spectrum of the category $\Mod_{A}(\cC)$ of $A$-modules internal to $\cC$ satisfies 
    \[
\pi_i(\pics(\Mod_A(\cC)) = \begin{cases}
\Pic(\Mod_A(\cC)) & i = 0 \\
(\pi_0(\Hom_{\cC}(\unit_{\cC},A))^{\times} & i = 1 \\
\pi_{i-1}(\Hom_{\cC}(\unit_{\cC},A)) & i > 1.
\end{cases}
  \]
  This follows because $A$ is the tensor unit in $\Mod_{A}(\cC)$. Indeed, writing $F \colon \cC \to \Mod_{\cC}(A)$ for the extension of scalars functor (so that $A \simeq F(\unit_{\cC})$), we have
  \[
  \begin{split}
\End_{\Mod_{A}(\cC)}(A)&= \Hom_{\Mod_{A}(\cC)}(F(\unit_{\cC}),A)\\
&\simeq \Hom_{\cC}(\unit_{\cC},A)
\end{split}
  \]
  by adjunction. 
\end{ex}
\subsection{Invertible objects in the \texorpdfstring{$\Sp_{k,n}$}{Spk,n}-local category}
Our main group of interest is the Picard group of $K(k) \vee \cdots \vee K(n)$-local spectra. 
\begin{defn}
  Let $\Pic_{k,n} = \Pic(\Sp_{k,n})$, the group of invertible $K(k) \vee \cdots \vee K(n)$-local spectra. 
\end{defn}
\begin{rem}\label{rem:pic_morphisms}
  By \cite[Lemma 2.2]{shimomura_picard} the localization functors induce natural morphisms $\Pic_{0,n} \to \Pic_{1,n} \to \cdots \to \Pic_{n,n}$.
\end{rem}

\begin{rem}
  The morphism $X \mapsto E_n \htimes X$ induces a functor $\Pic_{k,n} \to \Pic(\Mod_{E_n}(\Sp_{k,n}))$. We can fully understand the latter Picard group. 
\end{rem}
\begin{lem}\label{lem:pic_localized}
  For all $0 \le k \le n$ we have $\Pic(\Mod_{E_n}(\Sp_{k,n})) \cong \Pic(\Mod_{E_n}) \cong \Pic(E_{n_*}) \cong \bZ/2$. 
\end{lem}
\begin{proof}
We always have $\Pic(\Mod_{E_n}) \subseteq \Pic(\Mod_{E_n}(\Sp_{k,n}))$ because any invertible $E_n$-module is compact, and hence $E(n,J_k)$-local. The other inclusion follows if any $M \in \Pic(\Mod_{E_n}(\Sp_{k,n}))$ is compact as an $E_n$-module. Such an $M$ is automatically dualizable in $\Mod_{E_n}(\Sp_{k,n})$, and hence compact in $\Mod_{E_n}$ by \Cref{prop:dualizable_local}. This gives the first of the above isomorphisms, and the others hold by work of Baker and Richter \cite{baker_richter}. 
\end{proof}

We now give criteria for when $X \in \Pic_{k,n}$ is invertible. This (partially) extends work of Hopkins--Mahowald--Sadofsky \cite{hms_pic}, who considered the case $k = n$.
\begin{thm}\label{thm:picskn} Let $X \in \Sp_{k,n}$, then the following are equivalent. 
\begin{enumerate}
  \item $X \in \Pic_{k,n}$.
  \item $E_n \htimes X \in \Pic(\Mod_{E_n}(\Sp_{k,n}))$. 
  \item $E_n \htimes X \in \Pic(\Mod_{E_n})$
  \item $(E_{k,n})^\vee_*X \cong (E_n)_*$, up to suspension. 
\end{enumerate}
  \begin{proof}
    The equivalence of (1) and (2) follows from \Cref{cor:conservative,lem:invertible_conservative}, while the equivalence of (2) and (3) follows from \Cref{lem:pic_localized}, which also shows that (3) implies (4). Finally, to see that (4) implies (3), we note that if $M$ is any $E_n$-module whose homotopy groups are free of rank one over $(E_n)_*$, then $M$ is equivalent to a suspension of $E_n$ (for the elementary proof, see \cite[Proposition 2.2]{HeardStojanoska2014theory}). Thus, (4) implies that $E_n \htimes X \simeq E_n$, up to suspension, and hence (3) holds. 
  \end{proof}
\end{thm}

\begin{rem}
  When $n = 1$, there are two possibilities, the $K(1)$ and $E(1)$-local Picard groups, both of which are known. We record the results here:
  \[
\Pic_{0,1} = \begin{cases}
  \bZ \oplus \bZ/2 & p =2 \\
  \bZ  & p >2.
\end{cases}
\quad 
\quad 
\Pic_{1,1} = \begin{cases}
  \bZ_2 \oplus \bZ/4 \oplus \bZ/2 & p =2 \\
  \bZ_p \oplus \bZ/(p-1) \oplus \bZ/2 & p >2.
\end{cases}
  \]
  These are due to Hovey--Sadofsky \cite{HoveySadofsky1999Invertible} and Hopkins--Mahowald--Sadofsky \cite{hms_pic} respectively. 
  
  When $n = 2$, we have three possibilities, the $K(2)$, $K(1) \vee K(2)$, and $E(2)$-local Picard groups. The first and last are known for $p > 2$:
  \[
\Pic_{0,2} = \begin{cases}
  \bZ \oplus \bZ/3 \oplus \bZ/3 & p =3 \\
  \bZ  & p >3.
\end{cases}
\quad 
\quad 
\Pic_{2,2} = \begin{cases}
  \bZ_3 \oplus \bZ_3 \oplus \bZ/16 \oplus \bZ/3 \oplus \bZ/3 & p =3 \\
  \bZ_p \oplus \bZ_p \oplus \bZ/(2(p^2-1)) & p >3.
\end{cases}
  \]
These are due to a combination of authors: Hovey and Sadofsky \cite{HoveySadofsky1999Invertible}, Lader \cite{lader2013resolution}, Goerss--Henn--Mahowald--Rezk \cite{GoerssHennMahowaldRezk2015Hopkins}, Karamanov \cite{karamanov}, and Hopkins (unpublished). 

  This leaves the remaining case of $\Pic_{1,1}$. We note the following. 
\end{rem}
\begin{prop}
  Suppose $p \ge 3$, then the morphism $\Pic_{0,2} \to \Pic_{1,1}$ of \Cref{rem:pic_morphisms} is an injection. 
\end{prop}
\begin{proof}
  The morphism in question factors through the morphism $\Pic_{0,2} \to \Pic_{2,2}$ and so it suffices to show that this is an injection. When $p >3$ this is clear, and so we focus on the case $p = 3$. In this case, the calculations of Goerss--Henn--Mahowald--Rezk \cite{GoerssHennMahowaldRezk2015Hopkins} show that this map is an injection. 
\end{proof}
\begin{rem}
As noted in the proof the interesting case in the above propsition is the case $p = 3$. In fact, for all $n$ and $p \gg n$ we have that $\Pic_{0,n} \to \Pic_{i,n}$ is an injection for $i \ge 0$. However, here $\Pic_{0,n} \cong \bZ$ (by \cite{HoveySadofsky1999Invertible}), so this is not particularly helpful.   
\end{rem}
\subsection{Descent and Picard groups}
In \Cref{rem:pic_spectrum_limits} we recalled that we can associate a connective Picard spectrum $\pics(\cC)$ to a symmetric-monoidal $\infty$-category $\cC$. Using descent for the $E(n,J_k)$-local category, we now construct a spectral sequence whose $\pi_0$ computes $\Pic_{k,n}$. We need to introduce another algebraic gadget to describe the spectral sequence. 
\begin{defn}\label{def:unit}
  We let 
  $
  \Pic_{k,n}^{\mathrm{alg}}
  $
  denote the first cohomology of the complex
  \[
\xymatrix{(E_n)_0^{\times} \ar@<0.5ex>[r] \ar@<-0.5ex>[r]& ((E_{k,n})^\vee_0(E_n))^{\times} \ar@<1ex>[r] \ar@<-1ex>[r]  \ar[r]& \cdots }
  \]
  induced by taking the units in degree 0 in the cobar complex. 
\end{defn}
\begin{thm}\label{thm:pic_ss}
  There exists a spectral sequence with 
  \[
E_2^{s,t} \cong \begin{cases}
  \bZ/2 & s = t = 0 \\
    \Pic_{k,n}^{\mathrm{alg}} & s=t = 1 \\ 
  \widehat \Ext^{s,t-1}_{(E_{k,n})^{\vee}_*(E_n)}((E_n)_*,(E_n)_*) & t \ge 2
\end{cases}
  \]
  which converges for $t -s \ge 0$ to $\pi_{t-s}\pics(\Sp_{k,n})$. In particular, when $t = s$, this computes $\Pic_{k,n}$. The differentials in the spectral sequence run $d_r \colon E_r^{s,t} \to E_r^{s+r,t+r-1}$. 
\end{thm}
\begin{proof}
  Because $\pics$ commutes with limits (\Cref{rem:pic_spectrum_limits}), \Cref{thm:descendability_en} implies that
  \begin{equation}\label{eq:tot_module_pic}
\pics(\Sp_{k,n}) \simeq \tau_{\ge 0}\Tot\left(\xymatrix@C=1em{\pics(\Mod_{E_n}(\Sp_{k,n})) \ar@<0.5ex>[r] \ar@<-0.5ex>[r] & \pics(\Mod_{E_n \htimes E_n}(\Sp_{k,n}))\ar@<1ex>[r] \ar@<-1ex>[r] \ar[r]&}\right)
\end{equation}
We have (compare \Cref{ex:localized_pic})
\begin{equation}\label{eq:pic_en}
\pi_t(\pics(\Mod_{E_n^{\htimes i}}(\Sp_{k,n})) \cong \begin{cases}
  \Pic(\Mod_{E_n^{\htimes i}}) & t = 0 \\
  \pi_0(E_n^{\htimes i})^{\times} & t = 1 \\
  \pi_{t-1}(E_n^{\htimes i}) & t \ge 2. 
\end{cases}
\end{equation}
The Bousfield--Kan spectral sequence associated to \eqref{eq:tot_module_pic} has the form 
\[
E_2^{s,t} \cong H^s\left(\xymatrix@C=1em{\pi_t\pics(\Mod_{E_n}(\Sp_{k,n})) \ar@<0.5ex>[r] \ar@<-0.5ex>[r] & \pi_t\pics(\Mod_{E_n \htimes E_n}(\Sp_{k,n}))\ar@<1ex>[r] \ar@<-1ex>[r] \ar[r]&}\right)
\]
By \eqref{eq:pic_en} when $t \ge 2$, the spectral sequence is just a shift of the $E(n,J_k)$-local Adams spectral for $X = S^0$ sequence considered in \Cref{prop:anss}. 

When $t = 0$ and $i = 0$, we have $\Pic(\Mod_{E_n}(\Sp_{k,n})) \cong \bZ/2$ by \Cref{lem:pic_localized}. We do not know the higher terms, but this does not matter as only the $\bZ/2$ is relevant for the $s = t = 0$ part of the spectral sequence. 

Finally, we consider the $t = 1$ part of the spectral sequence. Again using \eqref{eq:pic_en}, we have 
\[
E_2^{s,1} \cong \xymatrix@C=1em{H^s ((E_n)_0^{\times} \ar@<0.5ex>[r] \ar@<-0.5ex>[r]& ((E_{k,n})^\vee_0(E_n))^{\times} \ar@<1ex>[r] \ar@<-1ex>[r]  \ar[r]& \cdots )}
\]
  By definition, when $s = 1$ this is $\Pic_{k,n}^{\mathrm{alg}}$. 
\end{proof}
\begin{rem}
 The proof shows that when $t = 1$, we can compute $E_2^{s,1}$ as the $s$-th cohomology of the complex in \Cref{def:unit}. However, unless $k = 0$ or $n$ we do not have a convenient description of this group (for the case $k = n$, see \Cref{ex:pic_k=n} below). 
\end{rem}
\begin{defn}
  We will say that $\Pic_{k,n}$ is algebraic if the only contributions come from the $s = 0$ and $s =1$ lines of the spectral sequence.
  \end{defn}
  \begin{rem}
    The $E_2^{0,0}$ term of the spectral sequence always survives the spectral sequence, as it is the Picard group of $E_n$-modules. It is however possible that there is a non-trivial differential in the $E_r^{1,1}$ spot. 
  \end{rem}
\begin{thm}\label{thm:algebraic}
  Suppose that $2p-2 \ge n^2+n-k$ and $p-1$ does not divide $k+s$ for $0 \le s \le n-k$, then $\Pic_{k,n}$ is algebraic. For example, this holds if $2p-2 > n^2+n$. 
\end{thm}
\begin{proof}
   For all primes $p$ and $t \ge 2$ we have that $E_2^{s,t} = 0$ unless $t-1 \equiv 0 \mod{2(p-1)}$ by \Cref{thm:sparseness}. In particular, for $s>2$, $E_2^{s,s} = 0$ unless $s \equiv 1 \mod{2p-2}$, and the lowest possible non-algebraic term is in filtration degree $2p-1$. 

  By \Cref{thm:vanishing} and the assumption that $p-1$ does not divide $k+s$ for $0 \le s \le n-k$ we have that $E_2^{s,s} = 0$ for $s > n^2+n-k$. Therefore, if additionally $2p-2 \ge n^2+n-k$ there can be no non-algebraic contributions to the spectral sequence. 

  Finally, if $2p-2 > n^2+n$, then $p > n+1$, and in particular $p-1$ does not divide $k+s$ for $0 \le s \le n-k$. 
  \end{proof}
\begin{ex}\label{ex:pic_k=n}
  Let us spell out the details in the case $k = n$. We first claim that the spectral sequence of \Cref{thm:pic_ss} takes the follows form: 
  \[
E_2^{s,t} \cong \begin{cases}
  \bZ/2 & s = t = 0 \\
  H^s_c(\bG_n,(E_n)_0^{\times}) & t = 1\\ 
  H^s_c(\bG_n,\pi_{t-1}E_n) & t \ge 2
\end{cases}
  \]
  and converges for $t -s \ge 0$ to $\pi_{t-s}\pics(\Sp_{K(n)})$. 

  The identification is much as in \Cref{rem:k(n)_local}. For the $t = 1$ term, we note that $\pi_0(E_n^{\htimes i})^{\times} \cong \Hom^c(\bG_n^{\times (i-1)},(E_n)_0)^{\times} \cong \Hom^c(\bG_n^{\times(i-1)},(E_n)_0^{\times})$. 

  The existence of such a spectral sequence is folklore, see \cite[Remark 6.10]{gh_duality} or \cite[Remark 2.6]{1811.05415}. In fact, the latter also proves \Cref{thm:algebraic} in the case $k = n$. 
\end{ex}

\section{\texorpdfstring{$E(n,J_k)$}{}-local Brown--Comenetz duality}\label{sec:bc_dual}
We recall the classical definition of Brown--Comenetz duality. The group $\mathbb{Q}/\mathbb{Z}$ is an injective abelian group, and so the functor 
\[
X \mapsto  \Hom(\pi_0X,\mathbb{Q}/\mathbb{Z})
\]
defines a cohomology theory on spectra represented by a spectrum $I_{\mathbb{Q}/\mathbb{Z}}$; this is the Brown--Comentz dual of the sphere. The Brown--Comenetz dual of a spectrum $X$ is then defined as $I_{\bQ/\bZ}X \coloneqq F(X,I_{\mathbb{Q}/\mathbb{Z}})$, and satisfies
\[
[Y,I_{\bQ/\bZ}X]_0 \cong \Hom(\pi_0(X \wedge Y),\bQ/\bZ). 
\]

It is an insight of Hopkins \cite{HopkinsGross1994rigid} that there is a good notion of Brown--Comenetz duality (also known as Gross--Hopkins duality) internal to the $K(n)$-local category, given by defining $I_nX = F(M_nX,I_{\mathbb{Q}/\mathbb{Z}})$ for a $K(n)$-local spectrum $X$. For details on this, see \cite{Strickl2000GrossHopkins}. As we will see, this definition can also naturally be made in the $E(n,J_k)$-local category. We begin with the following generalization of a result of Stojanoska \cite[Proposition 2.2]{MR2946825}. We recall that, by definition, $M_{0,n} = L_n$. In this case, the following lemma just says that $F(L_nX,Y)$ is already $L_n$-local. 
\begin{prop}\label{prop:localization}
  For any $X,Y$ the natural map $F(L_nX,Y) \to F(M_{k,n}X,Y)$ is $E(n,J_k)$-localization. 
\end{prop}
\begin{proof}
  We can repeat Stojanoska's argument. First, we show that $F(M_{k,n}X,Y)$ is $E(n,J_k)$-local. Indeed, let $Z$ be $E(n,J_k)$-acyclic, then we must show that 
  \[
F(Z,F(M_{k,n}X,Y)) \simeq F(Z \wedge M_{k,n}X,Y) \simeq F(M_{k,n}Z \wedge X,Y)
  \]
  is contractible. Here we have used that $M_{k,n}$ is smashing. But $M_{k,n}Z \simeq M_{k,n}L_{k,n}Z$ by \Cref{thm:category_equivalence} and this is contractible because $Z$ is $E(n,J_k)$-acyclic. 

  We now show that the fiber $F(L_{k-1}X,Y)$ is $E(n,J_k)$-acyclic. By \Cref{prop:hs_b.9} it suffices by a localizing subcategory argument to show this after smashing with a generalized Moore spectrum $M(k)$ of type $k$. Then (up to suspension), 
  \[
  \begin{split}
F(L_{k-1}X,Y) \wedge M(k) \simeq F(L_{k-1}X,Y) \wedge DM(k) &\simeq F(M(k),F(L_{k-1}X,Y)) \\ 
&\simeq F((L_{k-1}M(k))\wedge X,Y) \simeq \ast,
\end{split}
  \]
  as required. 
\end{proof}
\begin{defn}
  The $E(n,J_k)$-local Brown--Comenetz dual of $X$ is $I_{k,n}X = I_{\bQ/\bZ}(M_{k,n}X) = F(M_{k,n}X,I_{\bQ/\bZ})$. We let $I_{k,n}$ denote the $E(n,J_k)$-local Brown--Comenetz dual of $L_{k,n}S^0$. 
\end{defn}
\begin{rem}
It does not matter if we ask that $X$ be $E(n,J_k)$-local in the previous definition, as $I_{k,n}X$ only depends on the $E(n,J_k)$-localization of $X$. Indeed, we have equivalences
  \[
  I_{k,n}X = F(M_{k,n}X,I_{\bQ/\bZ}) \simeq F(M_{k,n}L_{k,n}X,I_{\bQ/\bZ}) = I_{k,n}(L_{k,n}X).
  \] 
  In particular, $I_{k,n} = I_{k,n}(L_{k,n}S^0) \simeq I_{k,n}S^0$.
\end{rem}
From the definition of $I_{\bQ/\bZ}$, we deduce the following. 
\begin{lem}
  There is a natural isomorphism
  \[
[Y,I_{k,n}X]_0 \cong \Hom(\pi_0(M_{k,n}(X) \wedge Y),\mathbb{Q}/\mathbb{Z}).
  \]
\end{lem}

As a consequence of \Cref{prop:localization} we deduce the following. 
\begin{lem}\label{lem:height_change}
  $I_{k,n}X$ is always $E(n,J_k)$-local. In fact, $I_{k,n}X \simeq L_{k,n}I_{\bQ/\bZ}(L_nX)$ and moreover, $I_{k,n}X \cong L_{k,n}I_{j,n}X$ for any $j \le k$.
\end{lem}
It follows that we have natural maps given by localization:
\[
I_{0,n} \to I_{1,n} \to \cdots I_{n,n}. 
\]
\begin{ex}\label{ex:bc_height1}
Let $n = 1$ and $p>2$, then $I_{0,1} \simeq L_1(S^2_p)$, the localization of the $p$-completion of $S^2$. On the other hand, when $p = 2$ we have $I_{0,1} \simeq \Sigma^2 L_1(DQ^\wedge_2)$ where $DQ$ is the dual question mark complex \cite[Remark 1.5]{MR1073009}. Similarly, $I_{1,1} \simeq L_{K(1)}S^2$ if $p>2$, while $I_{1,1} \simeq \Sigma^2 L_{K(1)}DQ$. 
\end{ex}
\begin{ex}\label{ex:bc_kn}
  We always have $I_{k,n}(K_n) \simeq K_n$. Indeed, first note that $L_{k,n}K_n \simeq K_n$, so \Cref{lem:height_change} allows us to reduce the case where $k = 0$ (although the proof is no more difficult in the other cases - just note that $M_{k,n}K_n \simeq K_n$).  Using the fact that 
  \[
  [Y,K_n]_* \cong   \Hom_{(K_n)_*}((K_n)_*X,(K_n)_*) 
  \]
  we argue as in \cite[Theorem 10.2(a)]{hovey_strickland99} to see that
  \[
[Y,K_{n}]_0 \cong \Hom((K_n)_0X,\mathbb{F}_p) \cong \Hom((K_n)_0X,\bQ/\bZ) \cong [Y,I_{0,n}(K_n)]_0.
  \]
This implies that $I_{0,n}(K_n) \simeq K_n$, as claimed. 
\end{ex}
  \begin{thm}\label{thm:bc_dual}
    Let $X \in \Sp_{k,n}$, then the natural map $X \to I^2_{k,n}X$ is an isomorphism when $\pi_*(F(k) \wedge X) \cong \pi_*(L_{k,n}F(k) \wedge X)$ is finite in each degree. In particular, this holds for $X = L_{k,n}S^0$. 
  \end{thm}
  \begin{proof}
    Let $\kappa_X \colon X \to I^2_{k,n}X$ denote the natural map. We first note that $I^2_{k,n}(F(k) \htimes X) \simeq I^2_{k,n}(F(k) \wedge X) \simeq F(k) \wedge I^2_{k,n}(X)$, because $F(k)$ is compact (and hence dualizable) in $\Sp$. As in \cite[Proof of Theorem 10.2]{hovey_strickland99} this identifies $\kappa_{F(k) \wedge X} \simeq \text{id}_F \wedge \kappa_X$, and so it is enough to show that $\kappa_{Y}$ is an equivalence, where $Y = F(k) \wedge X$. 

    Because $F(k)$ has type $k$, $L_{k-1}F(k) \simeq \ast$, and $M_{k,n}F(k) \simeq L_nF(k)$, so that $M_{k,n}Y = M_{k,n}(F(k) \wedge X) \simeq M_{k,n}F(k) \wedge X \simeq L_nF(k) \wedge X \simeq Y$. Likewise, $M_{k,n}(I_{k,n}Y) \simeq M_{k,n}(DF(k) \wedge I_{k,n}X) \simeq DF(k) \wedge I_{k,n}X \simeq I_{k,n}Y$. This implies that $\pi_*I^2_{k,n}Y \cong \Hom(\Hom(\pi_*Y,\bQ/\bZ),\bQ/\bZ)$, which is the same as $\pi_*Y$ because $\pi_*Y$ is finite in each degree. Therefore $\kappa_Y$ is an equivalence, as required. 
  \end{proof}
  \begin{rem}
  The Gross–-Hopkins dual $I_{n,n}$ is always an invertible $K(n)$-local spectrum. We do not know what happens for $I_{k,n}$ in general, however we note the following result.
  \end{rem}
\begin{prop}\label{prop:bc_equivalent}
  The following are equivalent:
  \begin{enumerate}
    \item $I_{k,n} \in \Sp_{k,n}^{\dual}$
    \item $I_{k,n} \in \Pic_{k,n}$
     \item $(E_{k,n})^\vee_*(I_{k,n})$ is a finitely generated $E_*$-module. 
    \item $E_n \htimes I_{k,n}$ is $K(n)$-local.
  \end{enumerate}
\end{prop}
\begin{proof}
Suppose first that $(1)$ holds. Then, $F(I_{k,n},I_{k,n}) \simeq DI_{k,n} \htimes I_{k,n}$, but on the other hand $F(I_{k,n},I_{k,n}) \simeq I_{k,n}^2(L_{k,n}S^0) \simeq L_{k,n}S^0$ by \Cref{thm:bc_dual}. It follows that $I_{k,n} \in \Pic_{k,n}$, i.e., that $(2)$ holds. The converse, $(2) \implies (1)$, always holds, see for example \cite[Proposition A.2.8]{HoveyPalmieriStrickl1997Axiomatic}. 

The equivalence of $(1)$ and $(3)$ is just \Cref{thm:dualspkn}. 

Finally to see that $(1) \iff (4)$ we note that it suffices to show that $(K_n)_*I_{k,n}$ is finite. In fact, because $(K_n)_*$ is a graded field, it suffices to see that $(K_n)^*I_{k,n}$ is finite. For this, we compute using \Cref{ex:bc_kn} and \Cref{thm:bc_dual}
  \[
[I_{k,n},K_n]_* \simeq [I_{k,n},I_{k,n}(K_n)]_* \simeq [I_{k,n},F(K_n,I_{k,n})]_* \simeq [K_n,F(I_{k,n},I_{k,n})]_* \simeq [K_n,L_{k,n}S^0]_*. 
  \]
  By \cite[Lemma 10.4]{hovey_strickland99} if $M_n$ denotes a generalized Moore spectrum of type $n$, then $[E_n,L_{K(n)}M_n]_* \simeq [E_n,L_{k,n}M_n]_*$ is finite (the last equivalence follows, for example, from the fact that $L_nM_n \simeq L_{K(n)}M_n$ for a generalized Moore spectrum of type $n$). As in \cite[Corollary 10.5]{hovey_strickland99} it follows that $[E_n \htimes DM_n,L_{k,n}S^0]$ is finite, and hence so is $[K_n,L_{k,n}S^0]$, as $K_n$ lies in the thick subcategory generalized by $E_n \htimes DM_n$ (note that $DM_n$ is also the localization of a generalized Moore spectrum of type $n$, cf.~\cite[Proposition 4.18]{hovey_strickland99}). 
\end{proof}
\begin{quest}\label{quest:bcdual}
  For which values of $k$ and $n$ do the conditions of \Cref{prop:bc_equivalent} hold? 
\end{quest}
\begin{rem}
  Condition (4) clearly holds in the case $k = n$. Of course, \Cref{prop:bc_equivalent} is precisely Hovey and Strickland's proof in this case. However, due to the $p$-completion this does not hold at $n =1$ and $k = 0$ (\Cref{ex:bc_height1}). In fact, this fails at all heights when $k = 0$, as we now explain.  
\end{rem}
\begin{rem}
Fix $n > 1,k=0$ and take $p \gg n$. Then $\Pic_{0,n} \cong \mathbb{Z}$, generated by $L_nS^1$ \cite[Theorem 5.4]{HoveySadofsky1999Invertible}. Therefore, if \Cref{prop:bc_equivalent} held for $k =0$ we must have $I_{0,n} \simeq L_nS^k$ for some $ k \in \mathbb{Z}$. On the other hand, work of Hopkins and Gross \cite{HopkinsGross1994rigid}, as written up by Strickland \cite{Strickl2000GrossHopkins}, and known results about the $K(n)$-local Picard group \cite[Proposition 7.5]{hms_pic} show that $I_{n,n} \simeq \Sigma^{n^2-n} S\langle \text{det} \rangle$, where $S \langle \text{det} \rangle$ is the determinant  sphere spectrum \cite{barthel2018constructing}. It cannot then be the case that $L_{K(n)}I_{0,n} \simeq I_{n,n}$, a contradiction to \Cref{lem:height_change}. We do not know what occurs in the cases $k \ne 0,n$. 
\end{rem}
\bibliography{bib_chromatic}
\bibliographystyle{amsalpha}
\end{document}